\documentclass[11pt,a4paper]{article}

\usepackage{amsthm,amsmath}
\usepackage{amssymb}
\usepackage{bbm}
\usepackage{upgreek}
\usepackage{mathtools}
\usepackage{graphicx}
\usepackage{enumerate}
\usepackage{caption}
\RequirePackage{natbib}
\RequirePackage[colorlinks,citecolor=blue,urlcolor=blue]{hyperref}

\numberwithin{equation}{section}
\theoremstyle{plain}

\theoremstyle{definition}

\theoremstyle{proposition}
\newtheorem{proposition}{Proposition}[section]

\theoremstyle{example}

\theoremstyle{lemma}
\newtheorem{lemma}{Lemma}[section]

\theoremstyle{assumption}
\newtheorem{assumption}{Assumption}[section]

\theoremstyle{theorem}
\newtheorem{theorem}{Theorem}[section]

\theoremstyle{remark}
\newtheorem*{remark}{Remark}

\newcommand{\E}{\operatorname{E}}

\newcommand\numberthis{\addtocounter{equation}{1}\tag{\theequation}}

\newenvironment{my_enum}{
\begin{enumerate}[(a)]
  \setlength{\itemsep}{1pt}
  \setlength{\parskip}{3pt}
  \setlength{\parsep}{3pt}
}{\end{enumerate}}

\title{Posterior consistency for the spectral density of non-Gaussian stationary time series}

\author{Yifu Tang\thanks{Department of Statistics, The University of Auckland, Auckland, New Zealand} \and Claudia Kirch\thanks{Institute for Mathematical Stochastics, Department of Mathematics, Otto-von-Guericke University, Magdeburg, Germany} \and Jeong Eun Lee\thanks{Department of Statistics, The University of Auckland, Auckland, New Zealand} \and Renate Meyer\thanks{Department of Statistics, The University of Auckland, Auckland, New Zealand}}

\date{}

\begin{document}
\maketitle


\begin{abstract}
Various nonparametric approaches for Bayesian spectral density estimation of stationary time series have been suggested in the literature, mostly based on the Whittle likelihood \citep{whittle} approximation. A generalization of this approximation
has been proposed in \citet{kirch_et_al} who prove posterior consistency for spectral density estimation in 
combination with the Bernstein-Dirichlet process prior for Gaussian time series.   In this paper, we will extend the posterior consistency result to non-Gaussian time series by employing a general consistency theorem of \citet{shalizi} for dependent data and misspecified models. As a special case, posterior consistency for the spectral density under the Whittle likelihood as proposed by \citet{choudhuri_et_al} is also extended to non-Gaussian time series. Small sample properties of this approach are illustrated with several examples of non-Gaussian time series.
\end{abstract}

\section{Introduction}
Many ecological, epidemiological, and physical time series show periodic behaviour. 
Spectral analysis of stationary time series in the frequency domain  is useful in detecting cycles and modelling the second order dependence structure of the data. 
However, a Bayesian approach generally requires an evaluation of the likelihood function. Even in the Gaussian case, the likelihood function can be computationally intractable due to the necessary inversion of a high-dimensional covariance matrix.
Therefore, most  Bayesian nonparametric methods for spectral density estimation (e.g.\ \citealt{Gango98,Liseo01,choudhuri_et_al,Hermansen08,Chopin13,Cadonna2017}) have used the Whittle likelihood approximation \citep{whittle} which was originally suggested as an approximation for Gaussian time series but later extended to non-Gaussian time series by \citet{10.2307/3212501}. Apart from avoiding matrix inversion, the Whittle likelihood has the advantage that it depends directly on the spectral density instead of indirectly via an inverse Fourier transform. However, using the Whittle likelihood approximation can be inefficient for non-Gaussian time series with moderate sample sizes \citep{a_note_on_whittle} and even for Gaussian time series with high autocorrelations and small sample sizes.
To address this issue, \citet{Chopin13} used importance sampling to correct a posterior sample based on the Whittle likelihood, \citet{Sykulski2017} introduced the de-biased Whittle likelihood,  \citet{RaoSuhasiniSubba2020RtGa} suggested  a new frequency-domain quasi-likelihood, and \citet{kirch_et_al} the so-called corrected likelihood. The latter is a generalization of the Whittle likelihood based on the idea of utilizing the efficiencies of a parametric working model but correcting this nonparametrically in the frequency domain to make it robust with respect to violations of the parametric  model assumptions.

When using any of these pseudo-likelihoods, one has to deal with model misspecification. An ensuing question is: When one adopts a pseudo-likelihood, will the posterior distribution associated with this pseudo-likelihood still be consistent in the sense that the posterior distribution eventually concentrates in the neighbourhood of the true parameter, namely the underlying spectral density function in this case, as the sample size gets large? For parametric Bayesian modelling using the Whittle likelihood, the Bernstein-von Mises theorem \citep{tamaki2008} guarantees that under some regularity conditions, Bayesian and frequentist inference procedures can produce the same asymptotic result. In the case of nonparametric Bayesian modelling, however, posterior consistency needs to be proved explicitly. While in the frequentist literature on nonparametric spectral density estimation (e.g. \citealt{pawitan_and_osullivan1994, kakizawa2006}) it is common to prove asymptotic results beyond the Gaussianity assumption, the same cannot be said for Bayesian nonparametric spectral density estimation. Prior works such as \cite{choudhuri_et_al} and \cite{kirch_et_al} established posterior consistency by utilizing the result that the Gaussian measure and the distribution associated with the pseudo-likelihood are mutually contiguous. However, in practice, it is often the case that the time series data at hand does not look anywhere close to a Gaussian time series. This calls for a theoretical justification of posterior consistency under mild assumptions on the time series without having to assume Gaussianity.

The difficulty of extending the consistency result to non-Gaussian time series is due to the fact that the contiguity argument, on which prior works relied heavily, is no longer viable for general non-Gaussian time series as shown in \citet[Section 4.3]{Alex-thesis}. Therefore, a novel method is needed to establish posterior consistency under mild conditions on the time series. To this end, we adopt the general consistency theorem of \citet{shalizi}.

The goal of this paper is to prove that under rather general conditions regarding the underlying time series, the posterior distribution associated with the corrected Whittle likelihood will still be consistent. As the Whittle likelihood is a special case of the corrected likelihood, the result will carry over to the posterior distribution based on the Whittle likelihood.

The article is organized as follows. A modification of the general consistency theorem of \citet{shalizi}, which is the main tool used to prove consistency in this paper, is introduced in Section \ref{section:2}. The assumptions on the time series, together with the working model, the parameter space and the prior are detailed in Section \ref{section:3}. In Section \ref{section:4} the conditions on time series are summarized with some examples. The proof of posterior consistency is illustrated in Section \ref{section:5}. Section \ref{section:6} discusses the tools and assumptions employed in this paper and outlines possible future directions. The proof of Theorem \ref{thm:2.1}, the modified theorem of \citet{shalizi}, is given in Appendix \ref{appendix:A}. The proof of Proposition \ref{proposition:3.1} as well as some results regarding Lipschitz norm and Bernstein polynomials are stated in Appendix \ref{appendix:B}. Finally, Appendix \ref{appendix:C} contains the proof of Proposition \ref{proposition:3.2}.

\section{Modification of the general consistency theorem of Shalizi}\label{section:2}
Historically, people tend to prove posterior consistency by gauging the discrepancy between the probability distribution of the sample and the 'proxy model' via Kullback-Leibler (KL) divergence. Examples are \citet{SchwartzLorraine1965OBp}, \citet{choudhuri_et_al}, \citet{choi_and_schervish}, \citet{kleijn_and_van_der_vaart} and \citet{shalizi}, to name but a few. Most of these papers require that the KL divergence be finite, which is reasonable as it ensures that even in the case of misspecification, the  discrepancy between the true distribution and the proxy model is not too severe. Nevertheless, an interesting question is: Will we obtain a consistent result even if we use some proxy model which drastically deviates from the true distribution in terms of the KL divergence? If the answer is affirmative, then is it possible to prove the consistency without resorting to the KL divergence?

Often, we just want to estimate the value of a parameter pertaining to the probability distribution of the sample. For instance, the parameter of interest may be the expectation, the variance-covariance matrix, the quantile function or the spectral density function. It is possible that the proxy model will result in a consistent posterior even though it deviates from the true distribution severely. We will use a rather simple example to illustrate this idea. Suppose we observe an i.i.d. real sequence $Y_1,Y_2,\cdots,Y_n$. We know nothing about the true probability distribution except $E|Y_1|<\infty$. The parameter of interest is the expectation $\mu_0=EY_1$. From a frequentist perspective, one can always use the arithmetic mean $\hat{\mu}_n=\frac{Y_1+\cdots+Y_n}{n}$ as the estimator because the Law of Large Numbers (LLN) tells us that $\hat{\mu}_n$ is always a consistent estimator no matter what the true probability distribution is. From a Bayesian point of view, however, it becomes complicated in that we have to adopt some (possibly) misspecified proxy model when using Bayes theorem in light of the absence of the true distribution. Suppose in this case we use the normal distribution with mean $\mu$ and variance 1 as proxy model and assume $\mu\sim N(0,1)$, i.e. the prior is the standard normal distribution. The resulting posterior distribution will be
$$
N\left(\frac{n}{n+1}\hat{\mu}_n, \frac{1}{n+1} \right).
$$
From the LLN as $n\rightarrow\infty$,  the posterior will concentrate on $\mu_0$, which is the true value of the parameter. In other words, posterior consistency is achieved by using a normal proxy model, even though this is not the actual likelihood associated with the sample. Even if the data come from the Cantor distribution (cf. e.g. \citet{cantor}), posterior consistency holds despite the fact that the normal proxy model is  not a good approximation in this case because the Cantor distribution does not even possess probability density function. In fact, the KL divergence between the Cantor distribution and the normal proxy model $N(\mu,1)$ is infinity for all $\mu\in\mathbb{R}$ due to the mutual singularity of the two probability measures.

There are some attempts to prove posterior consistency without invoking the KL divergence in the literature. For example, \citet{sriram_et_al} considered Bayesian quantile regression with proxy model being the asymmetric Laplacian distribution and proved consistency without evaluating the KL divergence between the true distribution and the proxy model when the prior is proper. When extending the result to an improper prior, however, they did use the KL assumption mainly to ensure that the posterior was a probability measure. Another example is \citet{syring_et_al} in the context of the quantile estimation in combination with a carefully chosen proxy model, where posterior consistency is achieved without using the KL assumption.

Despite the fact that \citet{shalizi} proved the consistency theorem by evaluating the KL divergence, those techniques  can be adapted so that the evaluation of the KL divergence is no longer required. In this section we derive a version of the consistency theorem of \citet{shalizi} that will then be used to prove posterior consistency for a spectral density estimator in Section~\ref{section:3}, where we
both strengthen and weaken certain assumptions.
For instance, apart from the avoidance of the KL assumption, we consider convergence in probability while \citet{shalizi} adopts almost sure convergence. Also, instead of considering a large parameter space, we confine ourselves to a smaller parameter set on which some desired properties hold.

Let $X_1,X_2, \cdots$, for short $X_1^{\infty}$, be a sequence of random variables on a probability space $(\Omega,\mathcal{F}, P)$, taking values in the measurable space $(\Xi, \mathcal{X})$. The sigma field generated by $X_1^n \coloneqq (X_1,X_2, \cdots, X_n)$ is denoted by $\sigma\left(X_1^n \right)$. Unless otherwise noted, all probabilities are taken with respect to $P$, and $\E[\cdot]$ always means expectation under this probability measure. 

We begin with a prior probability measure $\Pi_0$ over $(\Theta,\mathcal{T})$, and update this using Bayes rule after seeing $x_1^n=X_1^n(\omega)$. This updating is done based on a proxy model via a linking function 
$$
R_n\left(\theta,\omega \right) \coloneqq 
f_{\theta}\left(x_1^n \right),
$$
where $\theta $ is the parameter of interest taking values in the parameter space, a measurable space $(\Theta, \mathcal{T})$.  In classical Bayesian analysis the likelihood function $f_{\theta}(x_1^n)$ is used but here this is not required, rather a quasi-likelihood  or  some exponentiated loss or discrepancy function (see \citealt{bissiri_et_al,syring17}) can be used.
In the next section  we will use a generalization of the Whittle likelihood to this end.

In analogy to classical Bayesian statistics, for any $A \in \mathcal{T}$, we define the so-called (pseudo) posterior measure after seeing $x_1^n=X_1^n(\omega)$ to be
\begin{equation}
\Pi_n\left(A \right)\coloneqq\frac{\int_A R_n\left(\theta, \omega \right) \Pi_0\left( d\theta\right)}{\int_{\Theta} R_n\left(\theta, \omega \right) \Pi_0\left( d\theta\right)}=\frac{\int_A f_{\theta}\left(x_1^n \right) \Pi_0\left( d\theta\right)}{\int_{\Theta} f_{\theta}\left(x_1^n \right) \Pi_0\left( d\theta\right)}.\label{eq:2.1}
\end{equation}
The parameter space can either be finite-dimensional such as $\mathbb{R}^n$ or infinite-dimensional such as $C([0,1])$. 

In the following, we will often suppress the arguments $\theta$ and/or $\omega$ for brevity if this is possible without causing ambiguity.

The assumptions required for the proof of the main theorem, Theorem \ref{thm:2.1}, follow:

\begin{assumption}\label{ass:2.1}
For each $n\in\mathbb{N}$, the proxy model $R_n\left(\theta,\omega \right)$ is positive,  $\mathcal{T}\otimes\sigma\left(X_1^n \right)$ -measurable and,
$$
0<\int_{\Theta} R_n\left(\theta, \omega \right) \Pi_0\left( d\theta\right)<\infty \ \ P\textrm{-}a.s.
$$

\end{assumption}

As already stated we do not require $f_{\theta}(\cdot)$ to be the probability density function of some distribution, rather the inequalities in Assumption \ref{ass:2.1} ensures that the posterior measure $\Pi_n$ is well-defined.

Let $R_n(\theta)$ denotes the random variable $\omega\mapsto R_n(\theta,\omega)$. 

\begin{assumption}\label{ass:2.2}
For each $\theta \in \Theta$ and $n\in\mathbb{N}$, $\E|\ln ( R_n(\theta))|<\infty$ and, 
$$
h\left(\theta \right)\coloneqq -\lim_{n\rightarrow \infty}\frac{1}{n}\E(\ln( R_n(\theta))) 
$$
exists pointwise in $\theta$ such that $h$ is $\mathcal{T}$-measurable. Moreover, $|h(\Theta)|< \infty$, where $h(A)$ denotes the essential infimum of $h$ on $A\in\mathcal{T}$ with respect to $\Pi_0$.
\end{assumption}

The function $h(\cdot)$ defined in Assumption \ref{ass:2.2} can be negative. Note that the definition of $h(\theta)$ is different to  Assumption 2 of  \citet{shalizi} where $h(\theta)$ denotes the KL divergence rate from $P$. 

\begin{assumption}\label{ass:2.3}
As a function of $\omega$,
$$
\sup_{\theta\in\Theta}\left|\frac{1}{n}\ln (R_n\left(\theta,\omega \right))+ h\left(\theta \right)\right|
$$
is $\sigma\left(X_1^n \right)$ -measurable and,
$$
\sup_{\theta\in\Theta}\left|\frac{1}{n}\ln( R_n\left(\theta \right))+ h\left(\theta \right)\right| \overset{P}{\rightarrow}0, \ \ n \rightarrow \infty.
$$
\end{assumption}

Assumption \ref{ass:2.3}  effectively states a Uniform Weak Law of Large Numbers.

Now we state the modified general consistency theorem in which the proof is closely related to Theorem 3 of \citet{shalizi}, and is given in Appendix~\ref{appendix:A}.

\begin{theorem}\label{thm:2.1}
    Suppose that Assumptions \ref{ass:2.1} --\ref{ass:2.3} hold and in addition $h(A) > h(\Theta)$ for any $A \in \mathcal{T}$ such that $\Pi_0(A) > 0$. Then
$$
\Pi_{n}\left(A \right)\overset{P}{\rightarrow}0, \quad n \rightarrow \infty.
$$
\end{theorem}

The proof of the above theorem shows $P(\Pi_n(A)\ge \exp(-n\alpha))\to 0$ for some $\alpha>0$ as $n\to\infty$.

\section{Bayesian modelling and posterior consistency}\label{section:3}
In this section, we turn back to our original problem of interest and
introduce the assumptions on the observed time series, the proposed approximate likelihood function with its parameter space, as well as the nonparametric prior on the space of spectral density functions.

\subsection{Assumptions on the time series}\label{section:3.1}
Let $\{X_k \}_{k\in \mathbb{Z}}$ be a real-valued, zero-mean strongly stationary time series with finite variance on  $(\Omega,\mathcal{F},P)$. It is assumed that 
\begin{equation}
\E(X_0|\mathcal{U}_{-\infty})=0, \ \E(X_0X_k|\mathcal{U}_{-\infty})=\E(X_0X_k)\eqqcolon\gamma_0(k), \ \forall k \in \mathbb{Z}, \label{eq:3.1}
\end{equation}
where $\mathcal{U}_k=\sigma(X_s,s\leqslant k)$, $\mathcal{U}_{-\infty}=\cap^{\infty}_{k=-\infty}\mathcal{U}_k$ and $\sum_{k\in\mathbb{Z}}|\gamma_0(k)|<\infty$. Then, the spectral density exists, is continuous, $2\pi$-periodic and symmetric at 0, such that it is fully defined if known on the interval $[0,\pi]$. For convenience, we will rescale the spectral density to $[0,1]$, such that  
$$
\varphi_0(x)=\frac{1}{2\pi}\sum_{k\in\mathbb{Z}}\gamma_0(k)\exp\left(-i k\pi x \right)=\frac{1}{2\pi}\sum_{k\in\mathbb{Z}}\gamma_0(k)\cos\left( k\pi x \right), \quad x\in[0,1]\,.
$$
Here the second representation is due to the real and symmetric autocovariance function. The above assumptions also guarantee that a law of large numbers for the quadratic form $(X_1^n)^{\prime}A_n X_1^n$  for all $\|A_n\|_{s}\leqslant1$ (see Corollary 1 in \citet{yaskov}) where $X_1^n$ now denotes the column vector $(X_1,\ldots,X_n)^{\prime}$  and $\|\cdot\|_{s}$ denotes the matrix spectral norm.  This is closely related to spectral density estimation (see \cite{yaskov} for some further discussion) and will also be a key element of the below proof of consistency.

We further assume that
\begin{align}\label{eq_beta_null}
\varphi_0(x)\geqslant \beta \textrm{ for some }\beta>0,
\end{align}
and that $\varphi_0$ is Lipschitz continuous with Lipschitz constant $L(\varphi_0)$, i.e.
$$
L(\varphi_0)\coloneqq\sup_{x,y\in[0,1],x\neq y} \left|\frac{\varphi_0(x)-\varphi_0(y)}{x-y} \right|<\infty.
$$

The assumption of a Lipschitz continuous spectral density is stronger than absolute summability of the autocovariance sequence; see e.g. Bernstein's Theorem \citep[Theorem 6.3]{katznelson}.

However, it is weaker than the assumption made in \citet{choudhuri_et_al} and \citet{kirch_et_al}, that the auto-covariance function $\gamma_0(\cdot)$ satisfies
$$
\sum_{k\in\mathbb{Z}}|k|^{\alpha}|\gamma_0(k)|<\infty \textrm{ for some }\alpha>1.
$$
The latter implies that $\varphi_0$ is differentiable with $(\alpha-1)$-H\"older-continuous derivative (see e.g. Remark 8.5 in page 46 of \citet{serov}), which in turn implies Lipschitz continuity. This stronger assumption is required to show contiguity of the proxy and true likelihood functions, which cannot be expected to hold beyond Gaussian time series. Because we will prove consistency using Theorem~\ref{thm:2.1} above, we no longer require contiguity to hold.

We will present some examples of time series that satisfy all of our assumptions in Section~\ref{section:4}.

\subsection{Proxy model}\label{section:3.2}
The Whittle likelihood \citep{whittle} has been used in a variety of situations in both frequentist and Bayesian time series analysis. Whittle's likelihood assumes that the periodograms are exactly independent and exponentially distributed, which is well known to be true under quite general assumptions in some asymptotic sense. For applications concerning the spectral density alone this approximation is usually good enough, however, higher-order quantities (in the non-Gaussian case) often depend on the information coded in this dependency as observed by \citet{dahlhaus1996frequency}. 
To mitigate this effect \citet{kreiss2003autoregressive} and 
\citet{kirch_et_al} propose a related likelihood that inherits the dependency between periodogram ordinates from a parametric working model in the time domain, while at the same time `nonparametrically correcting' the inherited spectral density coded in the variances of the periodograms. 
In this context Whittle's likelihood becomes a `nonparametrically corrected' version of the Gaussian white noise working model.

Here, we use this corrected Whittle likelihood approximation with a Gaussian mean-zero parametric working model as proxy model.
We assume that the (rescaled) spectral density of the Gaussian parametric working model $\varphi_{\mathrm{par}}(x), \ x \in [0,1]$ is Lipschitz continuous and satisfies
$$
\varphi_{\mathrm{par}}(x)\geqslant \beta_{\mathrm{par}} \textrm{ for some }\beta_{\mathrm{par}}>0
$$
with autocovariance matrix $\Gamma_{\mathrm{par}}=\Gamma_{n,\mathrm{par}}=(\gamma_{\mathrm{par}}(k-j))_{k,j=1}^n$, where $\gamma_{\mathrm{par}}(\cdot)$ is the autocovariance function of the parametric working model.

The corrected Whittle likelihood is then defined as
\begin{equation}
f_{\varphi}(x_1^n)=(2\pi)^{-\frac{n}{2}}\det\left(\Gamma_{\mathrm{par}}^{-1} C_n^2(\varphi)\right)^{\frac{1}{2}}\exp\left(-\frac{1}{2}(x_1^n)^{\prime}C_n(\varphi)\Gamma_{\mathrm{par}}^{-1}C_n(\varphi)x_1^n \right)\label{eq:3.2},
\end{equation}
where $C_n\left(\varphi\right)
=F_n^{\prime}D_n^{-1/2}\left(\frac{\varphi}{\varphi_{\mathrm{par}}}\right)F_n$, $F_n$ is an orthogonal matrix called the real Discrete Fourier Transform matrix (see e.g. formula (4.5.4) of \citet{brockwell_and_davis} and Section 2.1 of \citet{kirch_et_al} for details), $\xi$ is a function on $[0,1]$ and $D_n(\cdot)$ the following corresponding diagonal matrix \begin{equation*}
D_n(\xi)=\begin{cases}
    \textrm{diag}\left(\xi\left(0\right),\xi\left(y_1\right),\xi\left(y_1\right),\cdots,\xi\left(y_N\right),\xi\left(y_N\right),\xi\left(y_{\frac{n}{2}}\right)\right), & \text{$n$ even}.\\
    \textrm{diag}\left(\xi\left(0\right),\xi\left(y_1\right),\xi\left(y_1\right),\cdots,\xi\left(y_N\right),\xi\left(y_N\right)\right), & \text{$n$ odd}.
  \end{cases}
\end{equation*}
where $y_j=\frac{2j}{n},\ j=1,\cdots,N=\left\lfloor\frac{n-1}{2}\right\rfloor$ are scaled Fourier frequencies. The unknown spectral density $\varphi$ is the parameter of interest.

The name {\em corrected} Whittle likelihood comes from the fact that a potentially misspecified parametric working model is nonparametrically corrected in the frequency domain by the matrix $D_n\left(\frac{\varphi}{\varphi_{\mathrm{par}}}\right)$. When a Gaussian white noise model is used as parametric working model, i.e.\ $\varphi_{\mathrm{par}}(x)$ is a constant function, then the corresponding corrected Whittle likelihood coincides with the original Whittle likelihood.

\subsection{Parameter space}

The true value of parameter is often assumed to lie in a compact set when proving consistency results. The parameter space, $\left(C([0,1]), \Vert \cdot \Vert_{\infty} \right)$, is the collection of all continuous functions on $[0,1]$ endowed with the sup-norm $\Vert\cdot \Vert_{\infty}$. It is well-known that the closed unit ball of $C([0,1])$ is not compact \citep[Theorem 12.10.2]{metric_space} and, compact subsets of $C([0,1])$ is obtained by using the following Lipschitz norm $\Vert \cdot \Vert_{\textrm{Lip}}$:

For any $\varphi\in C([0,1])$, the Lipschitz norm is defined by
$$
\Vert \varphi \Vert_{\textrm{Lip}}=\Vert\varphi \Vert_{\infty}+L(\varphi),
$$
with Lipschitz constant of $\varphi$,
$$
L(\varphi)=\sup_{0\leqslant x\neq y \leqslant1}\left|\frac{\varphi(x)-\varphi(y)}{x-y}\right|
\,.$$

The parameter space is defined as
\begin{align}\label{eq_param_space}
\Theta=\left\{\varphi\in C[0,1]: \Vert \varphi \Vert_{\textrm{Lip}}\leqslant M, \quad \inf_{0\leqslant x\leqslant 1}\varphi(x)\geqslant m \right\}
\end{align}
for some suitably chosen values $0<m<M<\infty$ and the true spectral density $\varphi_0$ is contained in $\Theta$. The choices for $m$ and $M$ will be shown later in Assumption~\ref{eq_ass_param_space}.

Let the metric $d_{\infty}(\cdot,\cdot)$ be induced by the sup-norm  $\Vert\cdot\Vert_{\infty}$, 
$$
d_{\infty}(\phi,\psi)=\Vert\phi-\psi \Vert_{\infty}=\sup_{x\in [0,1]}|\phi(x)-\psi(x)|
\,, \,\, \textrm{   for all }\phi,\psi\in\Theta.$$

 In this paper, we will work on the space $\left(\Theta, d_{\infty}(\cdot,\cdot)\right)$ equipped with the Borel $\sigma$-algebra
 $\mathcal{T}$.
 The following proposition shows that this space is a compact metric space and, the proof is given in Appendix~\ref{appendix:B}.

\begin{proposition}\label{proposition:3.1}
$\Theta$, equipped with $d_{\infty}(\cdot,\cdot)$, is a compact metric space.
\end{proposition}

\subsection{Bernstein-Dirichlet process prior}\label{section:3.4.1}
For the Bayesian analysis, we will use the Bernstein-Dirichlet process prior $\Pi_{BDP}$ proposed by \citet{choudhuri_et_al} for spectral density estimation.

As in
\citet{kirch_et_al} we will not put the prior directly on the spectral density $\varphi\in \Theta$ but rather on the correction $c(x)$, which relates to the spectral density via
\begin{align}\label{eq_connection_spectral_correction}
\varphi(x)=c(x){\varphi_{\mathrm{par}}^\delta(x)}
\end{align}
for some $0\leqslant \delta\leqslant 1$.
This implies that only for $\delta=0$  the Bernstein Dirichlet process prior is actually specified directly for the spectral density $\varphi$ (as in  \citep{choudhuri_et_al}), while for $\delta>0$ the spectral density of the parametric working model provides some pre-smoothing. The latter is important because Bernstein polynomials are better for smoother functions (see e.g. \citealt[Section 3.3]{bernstein} and \citealt[Appendix E]{nonpara_bayesian_inference}). Indeed for an (almost) correctly specified parametric working model a choice close to $\delta=1$ is ideal in terms of smoothness of the correction term while for a completely misspecified working model a choice of $\delta=0$ is better. For this reason, the parameter $\delta$ is called \emph{model confidence parameter}. In this work we assume $\delta$ to be fixed a priori. However, if a uniform prior is put on $\delta$, the MCMC algorithm implemented in the \textsf{beyondWhittle} R-package \citep{meier_et_al} does provide estimates of $\delta$ with values according to the above intuition (see the simulations in \citet{kirch_et_al}).

The Bernstein-Dirichlet process prior approximates the unknown correction function $c$ by Bernstein polynomials,
which  have a uniform approximation property (see e.g. \citealt[p.~20]{Kruijer}):
For any $c \sim \Pi_{BDP}$, define for $x\in[0,1]$
\begin{align}\label{eq_c}
&c(x)=\tau b(x;K,G),\\
&\text{with}\quad b(x;K,G)= b(x;K,w_K(G)):=\sum_{j=1}^{K}{w_{j,K}(G)\beta\left(x;j,K-j+1\right)},\notag
\end{align}
where
 $\beta\left(x;a,b\right)=\frac{\mathrm{\Gamma} (a+b)}{\mathrm{\Gamma}(a)\mathrm{\Gamma}(b)}x^{a-1}{(1-x)}^{b-1}$ are Beta-densities. 
 Priors are then put on both the variance parameter $\tau$,  the order of the polynomial $K$, as well as the weights of the Bernstein polynomial. The weight vector $w_K(G)=\left(w_{1,K}(G),\cdots,w_{K,K}(G) \right)$ is obtained from a Dirichlet process \newline $G \sim \textrm{ Dirichlet Process}(MG_0)$, where $M>0$ and $G_0$ is a probability distribution function on $[0,1]$ via $w_{j,K}(G)=G\left(\left.\frac{j-1}{K},\frac{j}{K}\right]\right.$,  $j=1,2,\cdots,K$. 
Samples from the Dirichlet process can be readily obtained using the truncated stick-breaking representation \cite{JayaramSethuraman1994ACDO}.

The Dirichlet process $G$, the order $K$ as well as the  variance parameter $\tau$ are assumed to be a priori independent, where $K$ has 
a probability mass function, $\rho(k)>0$, for $k=1,2,\cdots$, and $\tau$ is a continuous random variable with support $(0,\infty)$.

Some properties of Bernstein polynomials that are important for the proofs in this work are summarized in Appendix \ref{appendix:B}. 

Clearly, $c \sim \Pi_{BDP}$ puts an implicit prior denoted by $\Pi_C$ onto $\varphi$ via \eqref{eq_connection_spectral_correction}. We will further restrict this prior to the parameter set of interest $\Theta$
as follows:
$$
\Pi_0(A)\coloneqq\frac{\Pi_{C}(A\cap \Theta)}{\Pi_{C}(\Theta)},
$$	
for any $A\in\mathcal{T}$. 	

This is well defined by the next proposition (proven in Appendix~\ref{appendix:C}) which shows prior positivity in any ball around $\varphi_0$ for appropriately chosen $m,M$ in \eqref{eq_param_space}.
\begin{assumption}\label{eq_ass_param_space}
	Let $\inf_{0\leqslant x\leqslant 1} \varphi_0(x)>m$ and 
	$\left\|\frac{\varphi_0}{\varphi_{\mathrm{par}}^{\delta}}\right\|_{\mathrm{Lip}}\cdot \left\|\varphi_{\mathrm{par}}^{\delta}\right\|_{\mathrm{Lip}} <M$.
\end{assumption}

Assumption~\ref{eq_ass_param_space} guarantees that $\varphi$ as in \eqref{eq_connection_spectral_correction} with $c$ as in \eqref{eq_c} that are close enough to $\varphi_0$ in the supremum norm are also contained in $\Theta$. 

\begin{remark}\label{rem_param_space}\begin{enumerate}[(a)]
			\item For $\delta=0$ the second part of Assumption~\ref{eq_ass_param_space} clearly simplifies to $\|\varphi_0\|_{\mathrm{Lip}}<M$. 
			\item If we make the slightly stronger assumption that  $\varphi_0$ and $\varphi_{\mathrm{par}}$ are continuously differentiable, then the second assertion of Assumption~\ref{eq_ass_param_space} can be weakend to $\|\varphi_0\|_{\mathrm{Lip}}<M$ for any $0\leqslant \delta\leqslant 1$. For a proof we refer to Appendix~\ref{appendix:C}.
		\end{enumerate}
\end{remark}

\begin{proposition}\label{proposition:3.2}
	Under Assumption~\ref{eq_ass_param_space} it holds $\Pi_{C}(B(\varphi_0,r))>0$ for all $r>0$, where
\begin{equation}\label{eq:3.3}
B(\varphi_0,r)=\left\{\varphi\in \Theta: \left\Vert \varphi -\varphi_0\right\Vert_{\infty}< r \right\}.
\end{equation}
\end{proposition}
Thus $\Pi_0(\Theta)=1$ and $\Pi_0$ will assign positive probability to any neighbourhood of $\varphi_0$.

\subsection{Posterior consistency}

The posterior with respect to $\Pi_0$ and the proxy likelihood after updating $X_1^n$ is denoted by $\Pi_n$ and is defined in \eqref{eq:2.1}.

The following theorem is the main result of this paper, showing that with increasing sample size the posterior
$\Pi_n$ contracts towards a Dirac measure centred at the true spectral density $\varphi_0$. In particular, any (reasonable) point estimator obtained from the posterior will be a consistent estimator for the true spectral density even for non-Gaussian time series.

\begin{theorem}[Posterior consistency]\label{thm:3.1} 
    Under the assumptions on the time series and the prior in addition to Assumption~\ref{eq_ass_param_space}, 
for any $ r>0$,
$$
\Pi_n(\left\{\varphi\in\Theta: \Vert \varphi - \varphi_0 \Vert_{\infty} \geqslant r \right\})=\Pi_n(B^c(\varphi_0,r)) \overset{P}{\rightarrow} 0, \ t\rightarrow\infty,
$$
where $B^c(\varphi_0,r)$ denotes the complement of $B(\varphi_0,r)$ with respect to $\Theta$ and $B(\varphi_0,r)$ is as in \eqref{eq:3.3}.
\end{theorem}

Before we come to the proof of this theorem in Section~\ref{section:5}, we will present some non-Gaussian time series in which our posterior consistency result is applicable.

\section{Examples}\label{section:4}
To obtain posterior consistency, some assumptions on the underlying time series as outlined in Section~\ref{section:3.1} are required. 
A time series $\{X_n\}_{n\in\mathbb{Z}}$ is called regular \citep[Definition 2.13]{bradley2007introduction}
if its past tail $\sigma$-field $\cap^{\infty}_{n=-\infty}\sigma(X_k,k\leqslant n)$ is (almost) trivial. If the time series is regular, then \eqref{eq:3.1} automatically holds. \citet{peligrad_and_wu} presents several examples of regular processes, including mixing sequences, functions of Gaussian processes, functions of i.i.d. random variables and reversible Markov chains.
Apart from the above assumptions, we also require that the spectral density function of the time series $\varphi_0(\lambda), \ \lambda \in [0,1]$ is strictly positive and Lipschitz continuous. Consequently, any regular time series with strictly positive and Lipschitz continuous spectral density function satisfies all our assumptions. Below are two concrete examples of this category: The first one is non-Gaussian with a singular distribution and the second one has a Lipschitz continuous but not continuously differentiable spectral density. 

\subsection{Causal and invertible ARMA processes}\label{subsection 4.1}
	Causal and invertible ARMA$(p,q)$ model is powerful and popular in time series analysis. As a result of \citet[Theorem 5.22]{pourahmadi}, any such process $\{X_n\}_{n\in\mathbb{Z}}$ is regular. It is well-known that the spectral density function of $\{X_n\}_{n\in\mathbb{Z}}$ is strictly positive and continuously differentiable (cf. e.g. \citealt[Theorem 4.4.2]{brockwell_and_davis}).

\begin{figure}[!htb]\label{fig:Cantor1}
  \centering
  \includegraphics[width=0.8\textwidth]{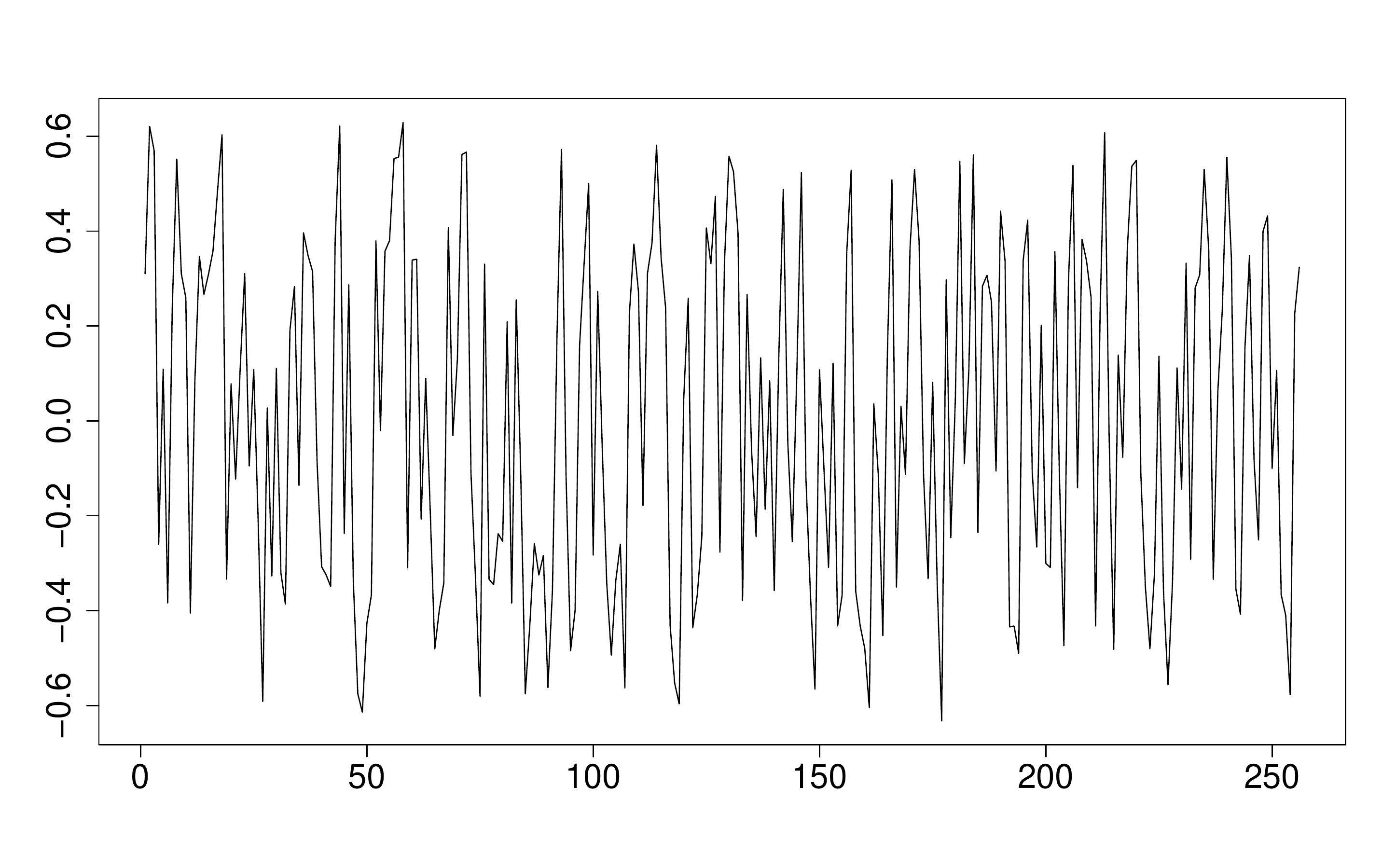}
  \caption{Simulated MA(1) time series with standard Cantor distributed innovations, sample size $n$=256.}
  \label{fig:5_1_trace}
\end{figure}
\begin{figure}[!htb]
  \includegraphics[width=0.8\textwidth]{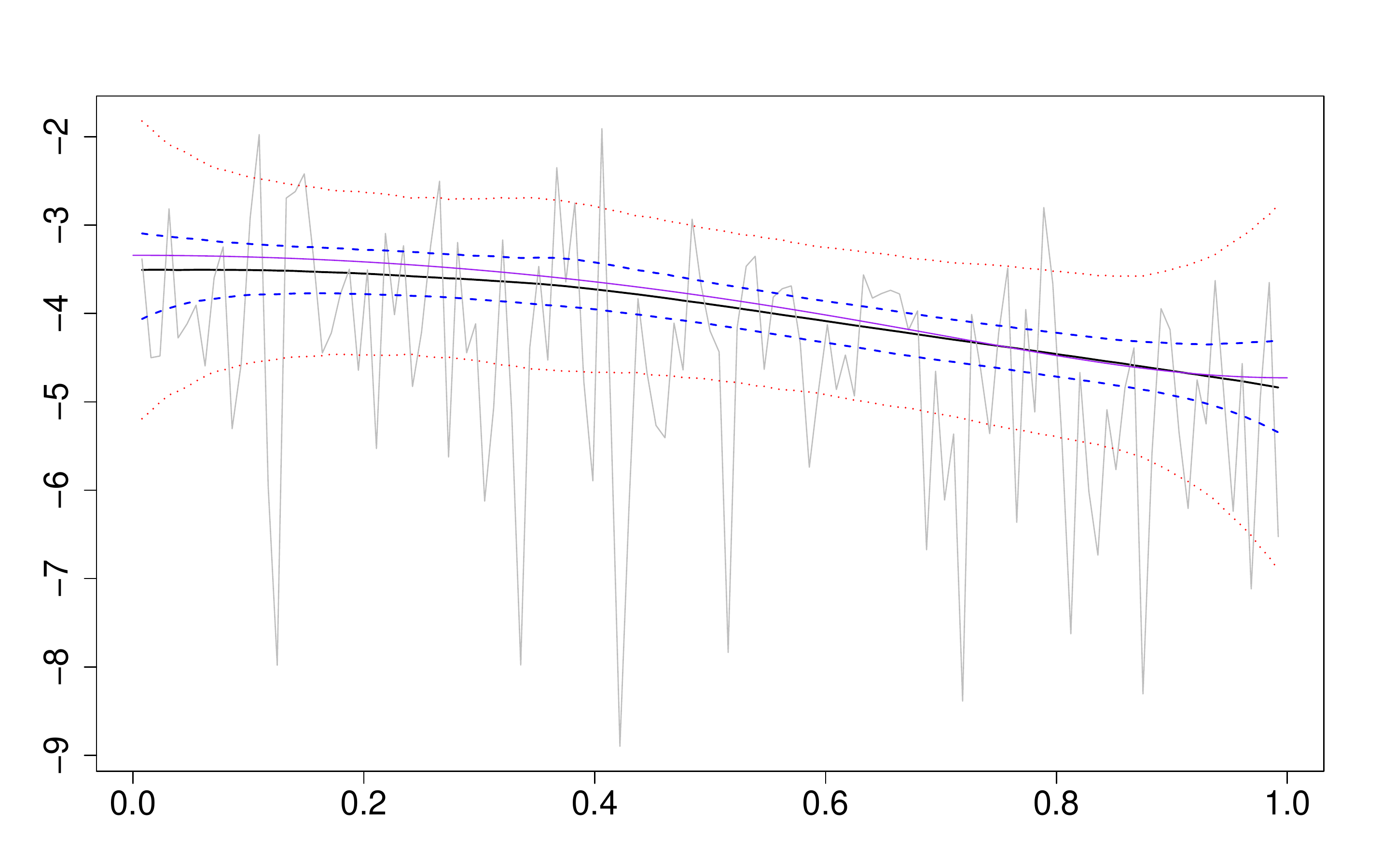}
  \caption{Log-periodogram of the time series in Figure 1, overlaid by the  true log-spectral density functions scaled to $[0,1]$ in solid purple, the estimated log-spectral density in solid black,  pointwise 90\% credible bands in dashed blue, and  uniform 90\% credible bands in dashed red.}
  \label{fig:5_1_spec}
\end{figure}

Some simulation results of posterior consistency for causal and invertible ARMA process with i.i.d non-Gaussian innovation are reported in the supplementary material of \citet{kirch_et_al}, indicating that posterior consistency does hold, which corroborates our finding. Here, we complement these finding with a more extreme innovation sequence, namely an MA$(1)$ process
with i.i.d. innovation obeying a centred standard Cantor distribution (see e.g. \citealt{cantor_function,cantor}):
\[
X_n=\epsilon_{n}+\frac{1}{3}\epsilon_{n-1}
\,, \hspace{.5cm} \epsilon_{k}=\sum_{n=1}^{\infty}\frac{2B_{n,k}}{3^n}-\frac{1}{2},
\]

where $\{B_{n,k}\}_{n,k=1}^{\infty}$ are i.i.d. Bernoulli random variables with expectation $1/2$.  The standard Cantor distribution has a cumulative distribution function which is continuous on $[0,1]$, but not absolutely continuous with respect to the Lebesgue measure \citep[Proposition 2.1]{cantor_function}. Therefore, it is a so-called singular distribution \citep[Remark 2.2]{cantor_function}, a property that is then inherited by $X_n$. Figure \ref{fig:5_1_trace} shows a realization of this time series with length $n=256$. In Figure~\ref{fig:5_1_spec}, the corresponding estimated spectral density function is represented by the black solid line, the blue dashed lines represent 90\% pointwise credible bands, the red dashed lines represent 90\% uniform credible bands and the purple solid line is the true spectral density function. The y-axis in Figure~\ref{fig:5_1_spec} is on the log scale. The corresponding Bayesian analysis via an MCMC algorithm as described in \citet{kirch_et_al} is conducted using the \textsf{beyondWhittle} R package \citep{meier_et_al}.

\subsection{Strong mixing sequences}
A stationary time series $\{X_t \}_{t\in \mathbb{Z}}$ is said to be strong mixing if 
$$
\lim_{n\rightarrow\infty}\alpha(\mathcal{F}_0,\mathcal{F}^n)=0
$$
where $\mathcal{F}_0=\sigma(X_k,k\leqslant0)$, $\mathcal{F}^n=\sigma(X_k,k\geqslant n)$, and $\alpha$ is the strong mixing coefficient defined by
$$
\alpha(\mathcal{A},\mathcal{B})=\sup\left\{|P(A\cap B)-P(A)P(B)|:A\in\mathcal{A},B\in\mathcal{B} \right\}
$$
for  two $\sigma$-algebras $\mathcal{A}$ and $\mathcal{B}$. 

\begin{figure}[!htb]
  \includegraphics[width=0.9\textwidth]{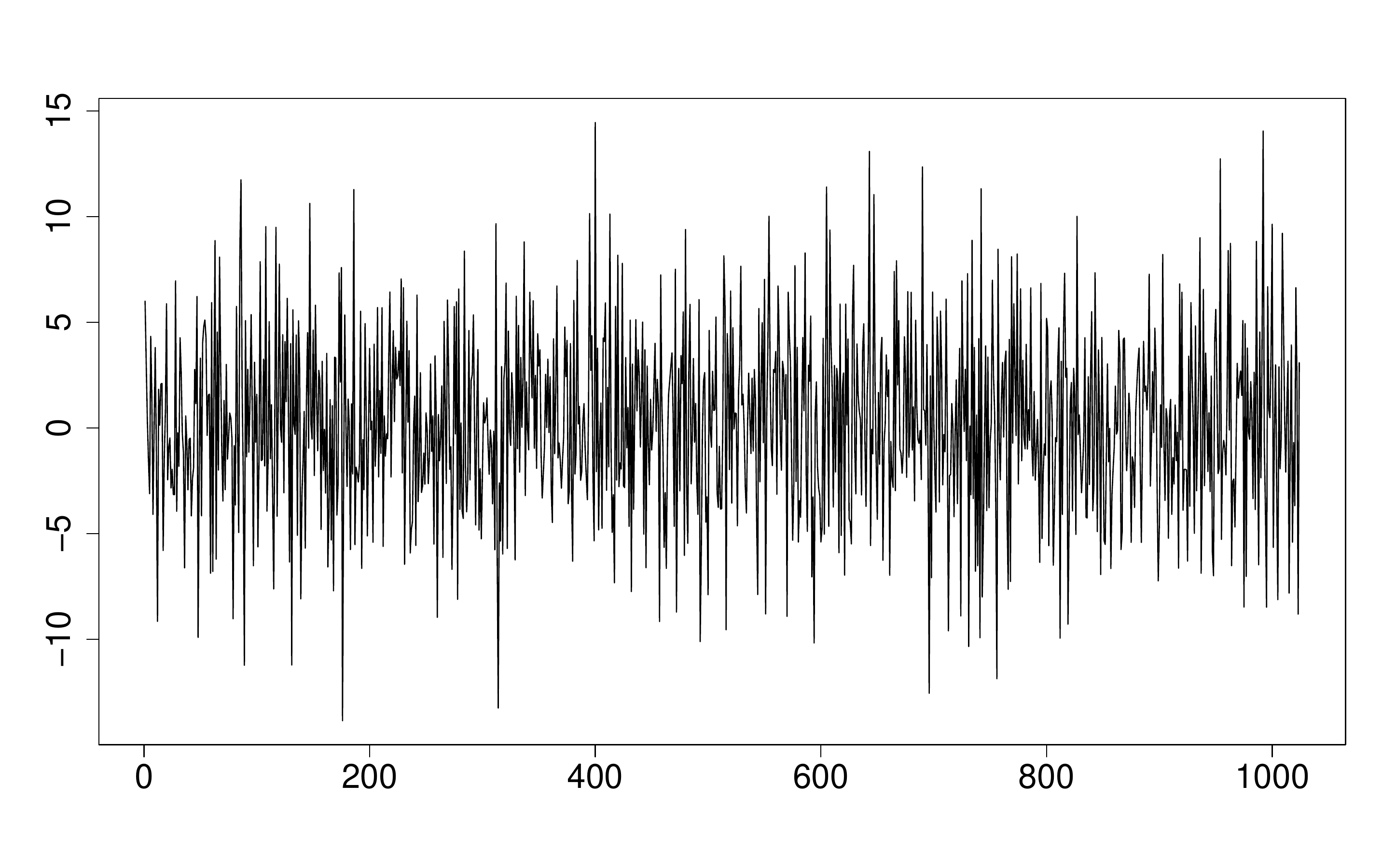}
  \caption{Realization of a Gaussian stationary time series with spectral density function (\ref{eq:psd0}) and sample size $n=1024$.}
  \label{fig:5_2_ts}
\end{figure}
\begin{figure}[h!]
  \includegraphics[width=0.9\textwidth]{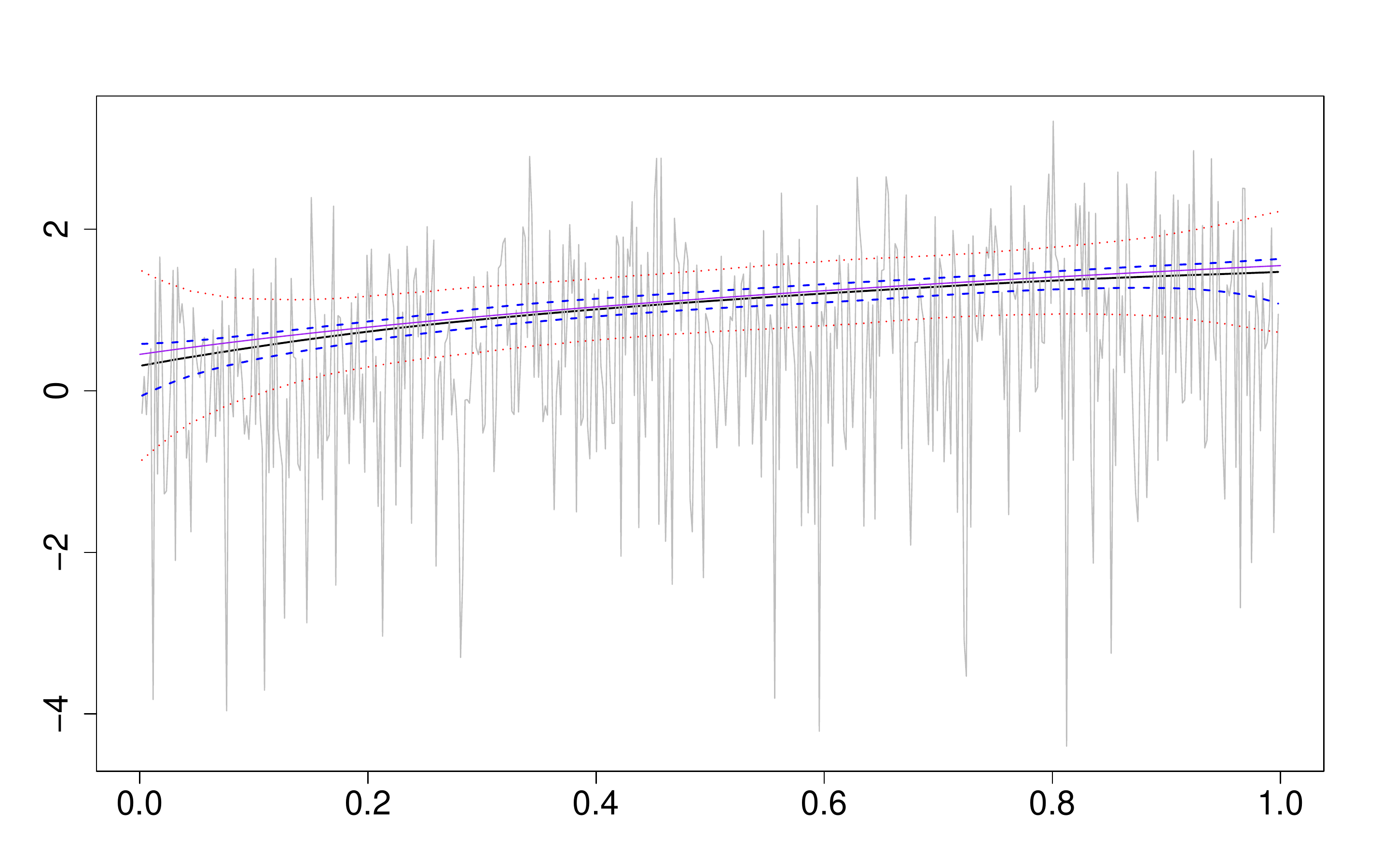}
  \caption{Log-periodogram of the time series in Figure 3, overlaid by true log-spectral density function scaled to $[0,1]$ in solid purple, its estimate in solid black, and pointwise 90\% credible bands in dashed blue, and uniform 90\% credible bands in dashed red.}
  \label{fig:5_2_spec}
\end{figure}

By \citet[Theorem 2.14]{bradley2007introduction}, any strong mixing sequence is also a regular sequence. Therefore, any strong mixing sequence with strictly positive and Lipschitz continuous spectral density function satisfies all our assumptions. Many time series models used in practice, including countable-state ergodic Markov chains \citep{bradley} and some of the ARCH  and GARCH models \citep[Section 3.2]{francq_and_zakoian}, are strong mixing. Specifically, a stationary Gaussian time series with strictly positive and Lipschitz continuous spectral density is also strong mixing \citep[Section 3, Chapter 17]{ibragimov_and_linnik}. Nevertheless, this property does not hold even for some AR(1) models with discrete innovations, see e.g.\ \citet{non_mixing_sequence}.

Figures \ref{fig:5_2_ts}  and \ref{fig:5_2_spec} show the trace plot, the spectral density function and its estimate for a mean-zero Gaussian process with spectral density function 
\begin{equation}\label{eq:psd0}
\varphi(x)=\pi|x|+\frac{\pi}{2}, \quad x\in[0,1].
\end{equation} 
The spectral density $\varphi(x)$ is chosen to be strictly positive and Lipschitz continuous but not continuously differentiable. Therefore, this time series satisfies our assumptions but not the assumptions in \citet{choudhuri_et_al} and \citet{kirch_et_al}. 
In fact, $\varphi_0$ admits the following representation.
$$
\varphi_0(\lambda)=\frac{1}{2\pi}\left(2\pi^2 - 8\sum_{n=1}^{\infty}\frac{\cos(2n-1)\pi x}{(2n-1)^2} \right).
$$
Hence, the auto-covariance sequence is
\begin{equation*}
  \gamma_0(h)=\begin{cases}
    2\pi^2, & \text{if $h=0$}.\\
    0, & \text{if $h$ is non-zero and even}. \\
    -\frac{4}{h^2}, & \text{if $h$ is odd}.
  \end{cases}
\end{equation*}
such that for all $\alpha>1$
$$
\sum_{h\in\mathbb{Z}}|h|^{\alpha}|\gamma_0(h)|=8\sum_{n=1}^{\infty}(2n-1)^{\alpha-2}=\infty.
$$

\section{Proof of the posterior consistency}\label{section:5}

We start with verification of Assumptions \ref{ass:2.1} -- \ref{ass:2.3} then show our main proof of
Theorem \ref{thm:3.1} using Theorem \ref{thm:2.1}.

\subsection*{Verification of Assumption \ref{ass:2.1}}
As a result of Proposition \ref{proposition:3.1}, $(\Theta,d_{\infty})$ is a separable metric space \citep[Lemma 3.26]{aliprantis_and_border}. $\mathcal{T}$ is the corresponding Borel $\sigma$-algebra on $\Theta$. Denote by $\mathcal{B}(\mathbb{R}^n)$ the Borel $\sigma$-algebra on $\mathbb{R}^n$. In order to prove that the proxy model $R_n(\varphi,\omega)=f_{\varphi}(x_1^n)$ defined in \eqref{eq:3.2} is $\mathcal{T}\otimes\sigma\left(X_1^n \right)$-measurable for all $n$, it suffices to prove that for any fixed $n\in\mathbb{N}$, $R_n(\varphi,\omega)$ is a Carath\'eodory function \citep[Lemma 4.51]{aliprantis_and_border}. That is, 
\begin{enumerate}[C.1]
	\item for each $\omega \in \Omega$, $R_n(\cdot,\omega)$ is a continuous functional on $\Theta$,
	\item for each $\varphi \in \Theta$, $R_n(\varphi,\cdot)$ is $\sigma\left(X_1^n \right)$-measurable.
\end{enumerate}
For any fixed $n$ and $\omega$ the corrected likelihood defined in \eqref{eq:3.2} is clearly a continuous functional in $\phi$ with respect to uniform convergence, confirming C.1.

Furthermore, as a composition of $X_1^n$ and an $n$-dimensional normal density function $f_{\varphi}$, clearly 
$R_n(\varphi,\cdot)$ is also $\sigma\left(X_1^n \right)$-measurable, showing C.2.

Next, we will show that the inequalities in Assumption \ref{ass:2.1} are satisfied: Clearly, for each $\varphi\in\Theta$ and each $n\in\mathbb{N}$, $R_n(\varphi,\omega)$ is $P$-a.s. strictly positive noting that $(X_1^n)^{\prime}X_1^n<\infty$ a.s.\ by the existence of second moments. This fact, along with Proposition \ref{proposition:3.2}, implies that
$$
\int_{\Theta} R_n\left(\theta, \omega \right) \Pi_0\left( d\theta\right)>0 \quad P\textrm{ -}a.s.
$$
Furthermore, because $\varphi\in\Theta$ has the uniform lower bound $m$, it holds that $\det(C^2_n(\varphi))\le \frac{1}{m} \det(D_n(\varphi_{\mathrm{par}}))<\infty$ such that
$$
\int_{\Theta} R_n\left(\theta, \omega \right) \Pi_0\left( d\theta\right)
<\infty  \quad P\textrm{ -}a.s.
$$
\subsection*{Verification of Assumption \ref{ass:2.2}}
We will first verify the integrability of $\ln R_n(\varphi)$. By direct calculation, we have
\begin{equation}
\ln R_n(\varphi,\omega) 
= -\frac{n\ln(2\pi)}{2}+\frac{1}{2}\ln\det\left(\Gamma_{\mathrm{par}}^{-1}C_n^2(\varphi) \right) -\frac{1}{2}(x_1^n)^{\prime}C_n(\varphi)\Gamma_{\mathrm{par}}^{-1}C_n(\varphi)x_1^n.\label{eq:4.1}
\end{equation}
Therefore, for each $\varphi\in\Theta$ and each $n\in\mathbb{N}$,
\begin{equation*}
\E|\ln R_n(\varphi)| 
\leqslant \frac{n\ln(2\pi)}{2}+\frac{1}{2}\left|\ln\det\left(\Gamma_{\mathrm{par}}^{-1}C_n^2(\varphi) \right)\right|+\frac{1}{2}\E\left|(X_1^n)^{\prime}C_n(\varphi)\Gamma_{\mathrm{par}}^{-1}C_n(\varphi)X_1^n\right|.
\end{equation*}
Because $C_n(\varphi)$ and $\Gamma_{\mathrm{par}}^{-1}$ are positive definite, we have 
\begin{align*}
\E\left|(X_1^n)^{\prime}C_n(\varphi)\Gamma_{\mathrm{par}}^{-1}C_n(\varphi)X_1^n\right|
&=\E\left[(X_1^n)^{\prime}C_n(\varphi)\Gamma_{\mathrm{par}}^{-1}C_n(\varphi)X_1^n\right]\\
&=\textrm{tr}\left(C_n(\varphi)\Gamma_{\mathrm{par}}^{-1}C_n(\varphi)\Gamma_n \right),
\end{align*}
where $\Gamma_n=\E\left[(X_1^n)(X_1^n)^{\prime}\right]$ is the autocovariance matrix of the time series.

We are now ready to calculate $h(\varphi)$. The functions $\varphi\in\Theta$, $\varphi_0$, $\varphi_{\mathrm{par}}$, $\sqrt{\frac{\varphi_{\mathrm{par}}}{\varphi}}$, $\frac{\varphi}{\varphi_{\mathrm{par}}}$ as well as their products are all Lipschitz continuous due to Proposition \ref{proposition:B.1}. Thus, by the Bernstein's Theorem \citep[Theorem 6.3]{katznelson} they are in the class of functions with absolutely summable Fourier (i.e.\ autocovariance) functions (Wiener algebra). Let $F$ be a continuous function, $f$  from the above class and $\tau_{n,k}$, $k=1,\cdots,n$, be the eigenvalues of an $n\times n$ matrix $T_n$. Furthermore, let $T_n\sim T_n(f)$ in the sense of Section 2.3 in  \citet{gray}, where $T_n(f)$ is the $n\times n$ Toeplitz matrix associated with $f$. Then, the fundamental eigenvalue distribution theorem of Szeg\"o (cf. e.g. \citealt[Theorem 4.2]{gray}), in conjunction with Theorem 2.4 of \citet{gray}, yields
$$
\lim_{n\to\infty}\frac{1}{n}\sum_{k=1}^{n}F(\tau_{n,k})=\int_{0}^{1}F(f(x))dx.
$$

Theorems 5.2 and 5.3 of \citet{gray} imply:
\begin{align*}
&\Gamma_{\mathrm{par}}C_n^{-2}(\varphi)\sim T_n\left(2\pi\varphi_{\mathrm{par}} \frac{2\pi\varphi}{2\pi\varphi_{\mathrm{par}}} \right)=T_n(2\pi\varphi),\\
&C_n(\varphi)\Gamma_{\mathrm{par}}^{-1}C_n(\varphi)\Gamma_n\sim T_n\left(\sqrt{\frac{2\pi\varphi_{\mathrm{par}}}{2\pi\varphi}}\frac{1}{2\pi\varphi_{\mathrm{par}}}\sqrt{\frac{2\pi\varphi_{\mathrm{par}}}{2\pi\varphi}}2\pi\varphi_0\right)=T_n\left(\frac{\varphi_0}{\varphi} \right).
\end{align*}
Thus, a choice of $F(x)=\ln x$ yields
$$
\lim_{n\rightarrow \infty}\frac{\ln\det \left(\Gamma_{\mathrm{par}}C_n^{-2}(\varphi)\right)}{n}= \int_{0}^{1}\ln \varphi(x)dx +\ln(2\pi),
$$
while choosing $F(x)=x$ yields
$$
\lim_{n\rightarrow \infty}\frac{1}{n}\textrm{tr}\left(C_n(\varphi)\Gamma_{\mathrm{par}}^{-1}C_n(\varphi)\Gamma_n \right)=\int_{0}^{1}\frac{\varphi_0(x)}{\varphi(x)}dx.
$$
Consequently, for each $\varphi\in\Theta$, we have by \eqref{eq:4.1}
\begin{align*}
h\left(\varphi \right)
&=-\lim_{n\rightarrow \infty}\frac{1}{n}\E\left[\ln R_n(\varphi) \right] \\
&=\ln(2\pi)+\frac{1}{2}\int_{0}^{1}\ln \varphi_0(x)dx+\frac{1}{2}\int_{0}^{1}\left(\frac{\varphi_0(x)}{\varphi(x)}-\ln\frac{\varphi_0(x)}{\varphi(x)} \right)dx. \numberthis \label{eq:4.2}
\end{align*}
Next, we show that $h$ is a continuous functional on $\Theta$ by verifying the Lipschitz condition. For any $0<u<v$, by the mean value theorem, there exists $w\in(u,v)$ such that
$$
|(u-\ln u)- (v-\ln v)|=\left|1-\frac{1}{w} \right||u-v|\leqslant \frac{v+1}{u}|u-v|.
$$

For all $\xi\in\Theta$, we have $\frac{m}{M}\leqslant\frac{\varphi_0(x)}{\xi(x)}\leqslant \frac{M}{m}$ for all $0\leqslant x\leqslant 1$. Therefore, for any $\phi,\xi\in\Theta$
\begin{align*}
|h(\xi)-h(\phi)|
&\leqslant \int_{0}^{1} \left|\left(\frac{\varphi_0(x)}{\xi(x)}-\ln\frac{\varphi_0(x)}{\xi(x)} \right) - \left(\frac{\varphi_0(x)}{\phi(x)}-\ln\frac{\varphi_0(x)}{\phi(x)} \right) \right|dx \\
&\leqslant \frac{\frac{M}{m}+1}{\frac{m}{M}}\int_{0}^{1} \varphi_0(x)\left|\frac{1}{\xi(x)}-\frac{1}{\phi(x)} \right|dx \\
&=\frac{ \frac{M}{m}+1}{\frac{m}{M}}\int_{0}^{1} \frac{\varphi_0(x)}{\xi(x)\phi(x)}|\xi(x)-\phi(x)|dx \\
&\leqslant \frac{\left(\frac{M}{m}+1\right)M}{\frac{m}{M}m^2} \Vert \xi-\phi \Vert_{\infty}.
\end{align*}
This indicates that $h$ is continuous on $\Theta$, which in turn implies $h$ is $\mathcal{T}$-measurable.

Finally, we will show that $|h|$ is bounded from above: Using the inequality $y-\ln y\geqslant 1$ on $y>0$, we have for all $\xi\in\Theta$,
$$
1\leqslant\int_{0}^{1}\left(\frac{\varphi_0(x)}{\xi(x)}-\ln\frac{\varphi_0(x)}{\xi(x)}\right)dx\leqslant \frac{M}{m}+\ln\frac{M}{m},
$$
such that
$$
\sup_{\xi\in\Theta}|h\left(\xi \right)|\leqslant \ln(2\pi)+\frac{1}{2}\left|\int_{0}^{1}\ln \varphi_0(x)dx\right|+\frac{1}{2}\left(\frac{M}{m}+\ln\frac{M}{m} \right)<\infty.
$$\medskip

\subsection*{Verification of Assumption \ref{ass:2.3}}
Let
$$
l_n(\varphi,\omega)\coloneqq \frac{1}{n}\ln R_n\left(\varphi,\omega \right)+ h\left(\varphi \right).
$$
We first show that for each $\varphi\in\Theta$, $l_n(\varphi) \overset{P}{\rightarrow}0$ as $n \rightarrow \infty$, and then strengthen it to uniform convergence.

Since for any $\varphi\in\Theta$, $\varphi$ and $\varphi_{\mathrm{par}}$ are both lower bounded away from 0 and upper bounded, the spectral norm of $C_n(\varphi)\Gamma_{\mathrm{par}}^{-1}C_n(\varphi)$ is upper bounded uniformly over $n\in\mathbb{N}$ \citep[Theorem 2.1 and Lemma 4.1]{gray}. Therefore, by Corollary 1 of \citet{yaskov}, we have
$$
\left|\frac{1}{n}\left[(X_1^n)^{\prime}C_n(\varphi)\Gamma_{\mathrm{par}}^{-1}C_n(\varphi)X_1^n \right]- \frac{1}{n}E\left[(X_1^n)^{\prime}C_n(\varphi)\Gamma_{\mathrm{par}}^{-1}C_n(\varphi)X_1^n \right]\right| \overset{P}{\rightarrow} 0.
$$
In view of equations \eqref{eq:4.1} and \eqref{eq:4.2}, the above result yields
$
\left|l_n(\varphi)\right| \overset{P}{\rightarrow}0.$

It is straightforward to show that $\sup_{\varphi\in\Theta}\left| l_n(\varphi,\omega)\right|$ is $\sigma\left(X_1^n \right)$ -measurable. Since $R_n(\varphi,\omega)$ is a Carath\'eodory function and $h\left(\varphi \right)$ is $\mathcal{T}$ -measurable, $\left|l_n(\varphi,\omega)\right|$ is also a Carath\'eodory function. According to the measurable maximum theorem \citep[Theorem 18.19]{aliprantis_and_border}, $\sup_{\varphi\in\Theta}\left|l_n(\varphi,\omega)\right|$ is $\sigma\left(X_1^n \right)$ -measurable.

We now show that 
$$
\sup_{\varphi\in\Theta}\left|l_n(\varphi)\right|\overset{P}{\rightarrow}0, \ n\rightarrow\infty.
$$
By Proposition \ref{proposition:3.1}, we know that $\Theta$ is compact with respect to $d_{\infty}$. According to \citet[Theorem 21.9]{davidson}, we can strengthen the pointwise convergence to uniform convergence by showing that $l_n(\varphi)$ is asymptotically uniformly stochastically equicontinuous. In order to prove equicontinuity, it suffices to show that $l_n(\varphi)$ satisfies a Lipschitz-type condition on $\Theta$ \citep[Theorem 21.10]{davidson}. That is, there exists $S\in\mathbb{N}$ such that for any $\phi,\psi\in\Theta$ and $n\geqslant S$,
\begin{equation}
\left|l_n(\phi) -l_n(\psi) \right|\leqslant B_n \Vert\phi-\psi \Vert_{\infty}, \label{eq:4.3}
\end{equation}
where $B_n=O_{p}(1)$.
Since the Lipschitz continuity of $h(\varphi)$ has been verified in the preceding subsection, we now prove that $\frac{1}{n}\ln R_n\left(\theta,\omega \right)$ satisfies \eqref{eq:4.3}.
In light of equation \eqref{eq:4.1}, we have
\begin{align*}
\frac{1}{n}\ln R_n\left(\theta,\omega \right)
&=-\frac{\ln(2\pi)}{2}  - \frac{1}{2}\frac{\ln\det \left(\Gamma_{\mathrm{par}}C_n^{-2}(\varphi)\right)}{n} - \frac{1}{2}\frac{(X_1^n)^{\prime}C_n(\varphi)\Gamma_{\mathrm{par}}^{-1}C_n(\varphi)X_1^n}{n} \\
&=:-\frac{\ln(2\pi)}{2} - \frac{1}{2} l_n^{(1)}(\varphi) - \frac{1}{2} l_n^{(2)}(\varphi,\omega). 
\end{align*}
It suffices to prove that $l_n^{(i)}$, $i=1,2$ satisfy the Lipschitz condition on $\Theta$.

For all $x,y>0$, as per the mean value theorem, there exists $z \in (\min\{x,y\},\max\{x,y\})$ such that
$$
|\ln x - \ln y|=\frac{1}{z}|x-y|\leqslant \frac{1}{\min\{x,y\}}|x-y|.
$$
Consequently, for any $\phi, \xi \in \Theta$, $n \in \mathbb{N}$,
\begin{align*}
&\left|l_n^{(1)}(\phi)-l_n^{(1)}(\xi)\right|
=\frac{1}{n}\left|\ln\det C_n^{-2}(\phi) - \ln\det C_n^{-2}(\xi) \right| =\frac{1}{n}\left|\sum_{j=1}^{N}\left(\ln\phi(y_j)-\ln\xi(y_j) \right) \right| \\
&\leqslant \frac{1}{n}\sum_{j=1}^{N}\frac{1}{\min\{\phi(y_j),\xi(y_j)\}}\left|\phi(y_j)-\xi(y_j)\right| \leqslant \frac{1}{2m}\max_{j=1,2,\cdots,N}\left|\phi(y_j)-\xi(y_j)\right| \\
&\leqslant \frac{1}{2m} \Vert \phi- \xi \Vert_{\infty}.
\end{align*}
Furthermore,
\begin{align*}
&\left|l_n^{(2)}(\phi)-l_n^{(2)}(\xi)\right|
=\frac{1}{n}\left|(X_1^n)^{\prime}\left(C_n(\phi)\Gamma_{\mathrm{par}}^{-1}C_n(\phi) - C_n(\xi)\Gamma_{\mathrm{par}}^{-1}C_n(\xi)\right)(X_1^n) \right| \\
&=\frac{1}{n}\left\Vert(X_1^n)^{\prime}\left(C_n(\phi)\Gamma_{\mathrm{par}}^{-1}C_n(\phi) - C_n(\xi)\Gamma_{\mathrm{par}}^{-1}C_n(\xi)\right)(X_1^n) \right\Vert_s \\
&\leqslant \Vert C_n(\phi)\Gamma_{\mathrm{par}}^{-1}C_n(\phi) - C_n(\xi)\Gamma_{\mathrm{par}}^{-1}C_n(\xi) \Vert_{s} \frac{\sum_{k=1}^n X_k^2}{n},
\end{align*}
where $\Vert \cdot \Vert_s$ denotes the matrix spectral norm and the last inequality is due to the submultiplicativity of the spectral norm (see e.g. \citealt[formula (2.3.3)]{matrix_comp}).
By the triangle inequality, we have
\begin{align*}
&\Vert C_n(\phi)\Gamma_{\mathrm{par}}^{-1}C_n(\phi) - C_n(\xi)\Gamma_{\mathrm{par}}^{-1}C_n(\xi) \Vert_{s} \\
&\leqslant \Vert (C_n(\phi)-C_n(\xi))\Gamma_{\mathrm{par}}^{-1}C_n(\phi)\Vert_{s} + \Vert C_n(\xi)\Gamma_{\mathrm{par}}^{-1}(C_n(\phi)-C_n(\xi)) \Vert_{s}.
\end{align*}
For the first term of the right hand side, again by the submultiplicativity of the spectral norm, we get
$$
\Vert (C_n(\phi)-C_n(\xi))\Gamma_{\mathrm{par}}^{-1}C_n(\phi)\Vert_{s}\leqslant \Vert C_n(\phi)-C_n(\xi) \Vert_s \Vert \Gamma_{\mathrm{par}}^{-1} \Vert_s \Vert C_n(\phi) \Vert_s.
$$
According to \citet[Theorem 2.1 and Lemma 4.1]{gray}, $\Vert \Gamma_{\mathrm{par}}^{-1} \Vert_s \leqslant \frac{1}{\beta_{\mathrm{par}}}$.
By direct calculation, we have
\begin{align*}
\Vert C_n(\phi)-C_n(\xi) \Vert_s 
&= \max_{j=1,\cdots,N}\left|\varphi_{\mathrm{par}}^{\frac{1}{2}}(y_j)\left(\phi^{-\frac{1}{2}}(y_j) - \xi^{-\frac{1}{2}}(y_j) \right) \right| \\
&=\max_{j=1,\cdots,N}\frac{\varphi_{\mathrm{par}}^{\frac{1}{2}}(y_j)}{\phi^{\frac{1}{2}}(y_j)\xi^{\frac{1}{2}}(y_j)\left(\xi^{\frac{1}{2}}(y_j) + \phi^{\frac{1}{2}}(y_j)\right)}\left|\xi(y_j) - \phi(y_j) \right| \\
&\leqslant \frac{\Vert \varphi_{\mathrm{par}}^{\frac{1}{2}} \Vert_{\infty}}{2m^{\frac{3}{2}}} \Vert \xi - \phi \Vert_{\infty}
\end{align*}
and
$$\Vert C_n(\phi) \Vert_s = \max_{j=1,\cdots,N} \frac{\varphi_{\mathrm{par}}^{\frac{1}{2}}(y_j)}{\phi^{\frac{1}{2}}(y_j)}\leqslant \frac{\Vert \varphi_{\mathrm{par}}^{\frac{1}{2}} \Vert_{\infty}}{m}.
$$
Therefore,
$$
\Vert (C_n(\phi)-C_n(\xi))\Gamma_{\mathrm{par}}^{-1}C_n(\phi)\Vert_{s}\leqslant \frac{\Vert \varphi_{\mathrm{par}} \Vert_{\infty}}{2\beta_{\mathrm{par}} m^{\frac{5}{2}}}\Vert \xi - \phi \Vert_{\infty}.
$$
By the same token, 
$$
\Vert C_n(\xi)\Gamma_{\mathrm{par}}^{-1}(C_n(\phi)-C_n(\xi))\Vert_{s}\leqslant \frac{\Vert \varphi_{\mathrm{par}} \Vert_{\infty}}{2\beta_{\mathrm{par}} m^{\frac{5}{2}}}\Vert \xi - \phi \Vert_{\infty}.
$$
Putting things together, we have
$$
\left|l_n^{(2)}(\phi)-l_n^{(2)}(\xi)\right|\leqslant \frac{\sum_{k=1}^n X_k^2}{n}  \frac{\Vert \varphi_{\mathrm{par}} \Vert_{\infty}}{\beta_{\mathrm{par}} m^{\frac{5}{2}}}\Vert \xi - \phi \Vert_{\infty}=O_p(1)\Vert \xi - \phi \Vert_{\infty},
$$
where the last equation is due to Corollary 1 of \citet{yaskov}.

Since $l_n^{(i)}$, $i=1,2$ satisfy the Lipschitz condition, we have shown that $l_n=-\frac{\ln(2\pi)}{2}- \frac{1}{2}l_n^{(1)} -\frac{1}{2} l_n^{(2)} + h$ does indeed satisfy the Lipschitz-type condition \eqref{eq:4.3}. Therefore, according to \citet[Theorem 21.10]{davidson}, $l_n$ is asymptotically uniformly stochastically equicontinuous in probability over $\Theta$.\medskip

\subsection*{Proof of posterior consistency}
Our goal is to show that for any $r>0$, $\Pi_n(B^c(\varphi_0,r))\rightarrow 0$ in $P$-probability. 
If $\Pi_0(B^c(\varphi_0,r))= 0$, then clearly we have $\Pi_n(B^c(\varphi_0,r))= 0$ for any $n\in\mathbb{N}$.
If $\Pi_0(B^c(\varphi_0,r))> 0$, then according to Theorem \ref{thm:2.1}, we only need to show
$$
h\left(B^c(\varphi_0,r) \right)=\underset{\varphi\in B^c(\varphi_0,r)}{\textrm{ess inf }}h(\varphi) > h(\Theta) = \underset{\varphi\in \Theta}{\textrm{ess inf }}h(\varphi),
$$
where here and throughout the essential infimum is with respect to the prior measure $\Pi_0$.

First, notice that the continuity of $h$ on $\Theta$ has been proved when verifying Assumption 2.2. We will now focus on proving that $h$ has a unique minimum point at $\varphi_0$. Since for any $y>0$, it holds $y-\ln y - 1\geqslant 0$ with equality if and only if $y=1$, and $\Theta$ is a collection of continuous functions on $[0,1]$, we have for any $\varphi\in\Theta$,
$$
\frac{\varphi_0}{\varphi}-\ln\frac{\varphi_0}{\varphi} - 1
$$
is a non-negative continuous function on $[0,1]$. Hence,
$$
\int_0^1 \left(\frac{\varphi_0(x)}{\varphi(x)}-\ln\frac{\varphi_0(x)}{\varphi(x)} - 1 \right)dx \geqslant 0,
$$
where the equality holds if and only if $\varphi=\varphi_0$. By virtue of equation \eqref{eq:4.2}, we get
$$
\varphi_0=\underset{\varphi\in\Theta}{\mathrm{argmin}} \ h(\varphi).
$$

Now we will prove that $h(\Theta)=h(\varphi_0)$. For any $z>h(\varphi_0)$, by the continuity of $h$, there exists $\delta>0$ such that $h(\varphi_0)\leqslant h(\varphi) < z$, for all $\varphi\in B(\varphi_0,\delta)$. Consequently, $\Pi_0(\left\{\varphi\in\Theta:h(\varphi)<z \right\})\geqslant\Pi_0(B(\varphi_0,\delta))>0$ by Proposition~\ref{proposition:3.2}  showing that  indeed $h(\Theta)=h(\varphi_0)$.

Finally, we will show that $h\left(B^c(\varphi_0,r) \right)> h(\Theta)$ for all $r>0$. It is clear that $B^c(\varphi_0,r)$, as a closed subset of a compact set $\Theta$  is also compact. Since $h$ is continuous on $B^c(\varphi_0,r)$, it can attain its local minimum on this set. i.e.
$$
h\left(B^c(\varphi_0,r) \right)\geqslant \min_{\varphi\in B^c(\varphi_0,r)}h(\varphi)>h(\varphi_0)=h(\Theta).
$$
Theorem~\ref{thm:3.1} now follows from Theorem~\ref{thm:2.1}.

\section{Discussion and outlook}\label{section:6}
The general consistency theorem introduced in this paper does not rely on the KL formulation which is otherwise popular in the literature. It has the capability to deal with situations where the proxy model does not approximate the data generating process sufficiently well with respect to the  KL divergence.  Examples of such time series with singular marginal distributions are presented in subsection \ref{subsection 4.1}. Since the proxy model in this case corresponds to a non-degenerate Gaussian probability measure, the Radon-Nikodym derivative and consequently the KL divergence between these two measures is not defined. On the other hand, these time series meet the  assumptions of this paper so that the proposed methodology  is  applicable. Moreover, the proxy model is not required to be tied to a probability distribution, it has the potential to deal with generalized Bayesian updating such as Gibbs posterior distributions.

The main objective in this paper was to make as weak assumptions on the time series as possible. Indeed posterior consistency is proven without invoking stronger assumptions such as the existence of higher order cumulants and certain complicated dependence structures which are commonly seen in the literature on nonparametric Bayesian (and frequentist) spectral density estimation (e.g. \citealt{choudhuri_et_al, pawitan_and_osullivan1994, kakizawa2006}). 

Mild assumptions on the time series structure, however, require stronger assumptions elsewhere. Similar to (3.2) of \citet{meier_et_al2020} and (15) of \citet{gao_and_zhou2016}, the parameter space in this paper is chosen to be compact under a pre-specified topology which is a common assumption. This restriction may be relaxed by using a sieve defined on the parameter space of positive continuous functions on [0,1] endowed with the sup-norm. Nevertheless, this comes at the cost of stronger assumptions on the time series. To elaborate, suppose we use $\Theta=\{f\in C([0,1]):f>0 \}$ as the parameter space and a suitable sequence of sets $\Theta_n=\{f\in C([0,1]):f\geqslant m_n, \left\Vert f \right\Vert_{\mathrm{Lip}}\leqslant M_n \}$ that approximate $\Theta$ well for growing $n$  where the covering number of  $\Theta_n$ needs to be controlled. The latter often relies on some sort of concentration inequality which in turn requires the existence of higher order moments and certain dependence structures (e.g. linear process, mixing condition, etc.). These are exactly the kind of assumptions which we aim to avoid. Such concentration inequalities are also typically required to obtain posterior contraction rates.

It is of future interest to establish a posterior contraction rate under minimal constraints on the time series structure. Another direction of future research, which may be of practical interest, is to generalize the results to priors other than the Bernstein-Dirichlet process prior used in this paper. As the proof  relies mainly on the uniform approximation property of the Bernstein polynomials, we conjecture that it can be generalized as long as the  basis functions have the uniform approximation property. Some examples are priors based on B-spline (\citealt{edwards_et_al2019} and \citealt{belitser_and_serra2014}) and wavelet basis \citep{gine_and_nickl2011} functions. Furthermore, our future research will also aim to extend the consistency proof to the estimation of the Hermitian positive definite spectral density matrix of a non-Gaussian multivariate time series using a matrix-Gamma process prior \citep{Alex-thesis}.

\appendix

\section{Proof of Theorem 2.1}\label{appendix:A}
The proofs in this section are closely related to the proofs for Lemmas 1, 2, 3, 6 and Theorem 1 in \citet{shalizi}.
\begin{lemma}\label{lemma_A1}
	Let $Q^n \subseteq \Theta \times \Omega$, $n=1,2,\cdots$ be a sequence of jointly measurable sets, with sections $Q_{\theta,\bullet}^n = \{\omega \in \Omega : (\theta, \omega) \in Q^n\}$ and $Q_{\bullet,\omega}^n = \{\theta \in \Theta : (\theta, \omega) \in Q^n\}$.

	If, for some probability measure $P$ on $\Omega$,
$$
\lim_{n \rightarrow \infty} P\left(Q_{\theta,\bullet}^n \right)=1, \ \forall \theta \in \Theta.
$$
Then for any probability measure $\Pi$ on $\Theta$,
$$
\Pi\left(Q_{\bullet,\omega}^n \right) \overset{P}{\rightarrow} 1.
$$
\begin{proof}
As in the proof of Lemma 1 in \citet{shalizi}, we obtain by Tonelli's theorem
$$
\int_{\Theta} P(Q_{\theta,\bullet}^n)\Pi(d\theta))=\E\Pi(Q_{\bullet,\omega}^n),
$$
such that by the dominated convergence theorem it holds
\begin{equation*}
\lim_{n \rightarrow \infty} \E|\Pi(Q_{\bullet,\omega}^n)-1|
=1-\int_{\Theta}\lim_{n \rightarrow \infty}P(Q_{\theta,\bullet}^n)\Pi(d\theta)
=0,
\end{equation*}
which implies
$$
\Pi\left(Q_{\bullet,\omega}^n \right) \overset{P}{\rightarrow} 1.
$$
\end{proof}

\end{lemma}

\begin{lemma}\label{lemma_A2}
	Under Assumptions \ref{ass:2.1}, \ref{ass:2.2} and \ref{ass:2.3}, for any $\eta>0$ it holds $\Pi_0\left(Q_{\bullet,\omega}^n(\eta) \right) \overset{P}{\rightarrow} 1$, where $Q_{\bullet,\omega}^n(\eta)$ is the $\omega$-section of
$$
Q^n(\eta)\coloneqq\left\{(\theta, \omega) \in \Theta \times \Omega:\frac{1}{n}\ln R_n(\theta,\omega) \geqslant -h(\theta)-\eta \right\}.
$$

\begin{proof}
For any $\eta >0$ and $n\in\mathbb{N}$,
$Q^n(\eta)$
is jointly measurable according to Assumptions \ref{ass:2.1} and \ref{ass:2.2}. Furthermore, Assumption \ref{ass:2.3} implies that for all $\theta\in\Theta$
$$
\lim_{n \rightarrow \infty}P\left(Q_{\theta,\bullet}^n (\eta) \right)=1,
$$
such that by Lemma \ref{lemma_A1}, we have
$$
\Pi_0\left(Q_{\bullet,\omega}^n(\eta) \right) \overset{P}{\rightarrow} 1.
$$
\end{proof}
\end{lemma}

\begin{lemma}\label{lemma_A3}
Under Assumptions \ref{ass:2.1}, \ref{ass:2.2} and \ref{ass:2.3}, for any $\epsilon > 0$,
$$
\lim_{n \rightarrow \infty} P\left(\left\{\omega\in\Omega: \Pi_0\left(R_n \right)> \exp\left\{-n\left(h\left(\Theta \right)+\epsilon\right)\right\} \right\}\right) = 1,
$$
where $\Pi_0\left(R_n \right)\coloneqq\int_{\Theta}R_n(\theta,\omega)\Pi_0(d\theta)$.
\begin{proof}
For any $\omega \in \Omega$, we have
\begin{equation*}
\Pi_0 (R_n) > \exp\left\{-n\left(h\left(\Theta \right)+\epsilon\right)\right\}
\Leftrightarrow
\int_{\Theta} \exp\left\{n\left(h\left(\Theta \right)+\epsilon+  \frac{\ln R_n(\theta)}{n}\right)\right\}  \Pi_0(d\theta) > 1.
\end{equation*}
Consequently, it is sufficient to prove
\begin{align}\label{eq_lemma_a3_1}
\lim_{n \rightarrow \infty} P\left(\left\{\omega\in\Omega: \int_{\Theta} \exp\left\{n\left(h\left(\Theta \right)+\epsilon+  \frac{\ln R_n(\theta)}{n}\right)\right\}  \Pi_0(d\theta) > 1 \right\}\right)=1.
\end{align}
For $N_{\frac{\epsilon}{2}}=\{\theta: h(\theta) \leqslant h(\Theta)+\frac{\epsilon}{2} \}$ it holds for any $\theta\in N_{\frac{\epsilon}{2}}\cap Q_{\bullet,\omega}^n\left(\frac{\epsilon}{3}\right)$  by definition of these sets
$$
\frac{\ln R_n(\theta,\omega)}{n} \geqslant -h(\theta)-\frac{\epsilon}{3}\geqslant -h(\Theta)-\frac{5}{6}\epsilon.
$$
Consequently,
\begin{align}\label{eq_lemma_a3_2}
	& \int_{\Theta} \exp\left\{n\left(h\left(\Theta \right)+\epsilon+  \frac{\ln R_n(\theta)}{n}\right)\right\}  \Pi_0(d\theta)\geqslant \Pi_0\left(N_{\frac{\epsilon}{2}}\cap Q_{\bullet,\omega}^n\left(\frac{\epsilon}{3}\right) \right)\exp\left\{\frac{\epsilon}{6}n\right\}
\end{align}According to Lemma \ref{lemma_A2},
$$
\Pi_0\left(\Theta \setminus Q_{\bullet,\omega}^n\left(\frac{\epsilon}{3}\right) \right) \overset{P}{\rightarrow} 0
$$
and by definition that $\Pi_0\left( N_{\frac{\epsilon}{2}}\right)>0$, such that
$$
\Pi_0\left(N_{\frac{\epsilon}{2}}\cap Q_{\bullet,\omega}^n\left(\frac{\epsilon}{3}\right) \right) \geqslant \Pi_0\left(N_{\frac{\epsilon}{2}}\right) - \Pi_0\left(\Theta \setminus Q_{\bullet,\omega}^n\left(\frac{\epsilon}{3}\right) \right) \overset{P}{\rightarrow} \Pi_0\left(N_{\frac{\epsilon}{2}}\right)>0.
$$
This concludes the proof in combination with \eqref{eq_lemma_a3_1} and \eqref{eq_lemma_a3_2}.
\end{proof}
\end{lemma}

\begin{lemma}\label{lemma_A4}
	Let Assumptions \ref{ass:2.1}, \ref{ass:2.2} and \ref{ass:2.3} hold, and take any set $G\in\mathcal{T}$ such that $\Pi_0(G) > 0$. Then for any $\epsilon >0$,
$$
\lim_{n \rightarrow \infty} P\left(\left\{\omega\in\Omega: \Pi_0\left(G R_n \right)\leqslant \exp\left\{n\left(-h(G)+\epsilon\right)\right\}\right\}\right) = 1,
$$
where $\Pi_0\left(GR_n \right)=\int_{G}R_n(\theta,\omega)\Pi_0(d\theta)$.
\begin{proof}
For any $\omega\in\Omega$ and sufficiently large $n$ such that 
$$
\sup_{\theta\in G}\left|\frac{1}{n}\ln\left(R_n(\theta,\omega)\right)+h(\theta) \right|  \leqslant \epsilon,
$$
we have
\begin{align*}
\Pi_0\left(GR_n \right)
&=\int_{G}\exp\{\ln(R_n(\theta,\omega))\}\Pi_0(d\theta)\leqslant \int_{G}\exp\left\{n\left(-h(\theta)+\epsilon \right)  \right\}\Pi_0(d\theta) \\
&\leqslant \exp\left\{n\left(-h(G)+\epsilon\right)\right\} \,\Pi_0(G)\leqslant \exp\left\{n\left(-h(G)+\epsilon\right)\right\} .
\end{align*}
\end{proof}

\end{lemma}

With these lemmas, we can proceed to proving Theorem~\ref{thm:2.1}. 
\begin{proof}[Proof of Theorem~\ref{thm:2.1}]

By the assumption that $h(A)>h(\Theta)$, we can choose $\epsilon$ such that $h(\Theta)-h\left(A \right)+2\epsilon < 0$.
Then, on the intersection of the asymptotic 1-sets of Lemmas~\ref{lemma_A3} and \ref{lemma_A4} it holds
\begin{align*}
	&\Pi_n(A)=\frac{\Pi_0(AR_n)}{\Pi_0(R_n)}\leqslant \exp\left\{n\left(h\left(\Theta \right) - h(A) + 2\epsilon\right) \right\}\to 0,
\end{align*}
such that indeed $\Pi_n\left(A \right)\overset{P}{\rightarrow} 0$ for $n\to\infty$.
\end{proof}

\section{Parameter space, Lipschitz functions and Bernstein polynomials}\label{appendix:B}
In this section, we will first prove Proposition \ref{proposition:3.1}, and then give some results regarding Lipschitz functions and Bernstein polynomials.

\begin{proof}[Proof of Proposition \ref{proposition:3.1}]
According to the Arzel\`a-Ascoli theorem, the compactness of $\Theta$ is guaranteed if it is bounded, closed, and equicontinuous. The boundedness of $\Theta$ with respect to $\Vert\cdot \Vert_{\infty}$ follows directly from the definition of $\Theta$. The equicontinuity follows immediately from the uniform boundedness of the Lipschitz constants over $\Theta$. 

To prove that $\Theta$ is closed, consider an arbitrary sequence $\{\varphi_n \}_{n=1}^{\infty}$  of functions in $\Theta$ which converge uniformly to $\varphi\in C([0,1])$. It is then sufficient to show that $\varphi\in\Theta$. By the triangle inequality it follows from $\Vert \varphi_n - \varphi \Vert_{\infty}\rightarrow0$, that $\Vert \varphi_n \Vert_{\infty}\rightarrow\Vert \varphi \Vert_{\infty}$. 
Thus, for any $x\in[0,1]$,
$$
\varphi(x)=\lim_{n\rightarrow\infty}\varphi_n(x)\geqslant m.
$$
Furthermore,
\begin{align*}
	&\|\varphi\|+L(\varphi)=\lim_{n\to\infty}\|\varphi_n\|+\sup_{0\leqslant x\neq y\leqslant 1}\lim_{n\rightarrow\infty}\frac{|\varphi_n(x)-\varphi_n(y)|}{|x-y|} \\
	&=\sup_{0\leqslant x\neq y\leqslant 1}\lim_{n\to\infty}\left( \|\varphi_n\|+ \frac{|\varphi_n(x)-\varphi_n(y)|}{|x-y|} \right)\leqslant \limsup_{n\to\infty}\Vert \varphi_n\Vert_{\mathrm{Lip}} \leqslant M.
\end{align*}
\end{proof}\smallskip

\subsection*{Results regarding Lipschitz norm and the Bernstein polynomials}
Let 
$$
\textrm{Lip}([0,1])\coloneqq \left\{\varphi\in C([0,1]): \Vert \varphi \Vert_{\textrm{Lip}}<\infty \right\}.
$$
The following inequalities regarding the Lipschitz norm will be useful in the proof of positivity of prior probabilities in Appendix \ref{appendix:C}.
\begin{proposition}\label{proposition:B.1}\ 
	\begin{enumerate}[(a)]
	\item If $f\in \mathrm{Lip}([0,1])$ and $f>0$, then we have  $\frac{1}{f},f^{\delta}\in \mathrm{Lip}([0,1])$ for any $0\leqslant \delta\leqslant 1$.\label{proposition:B.1(a)}
	\item For any $f,g\in \mathrm{Lip}([0,1])$, we have $fg \in\mathrm{Lip}([0,1])$. Furthermore, $\|fg\|_{\mathrm{Lip}}\leqslant \|f\|_{\mathrm{Lip}}\|g\|_{\mathrm{Lip}}$.\label{proposition:B.1(b)}
\end{enumerate}
\begin{proof}
The proof of (a) is a simple consequence of Theorems 12.5 and 12.6 of \citet{applied_math_body_and_soul}. The proof of (b) can be found in \citet[Proposition 11.2.1]{dudley}
\end{proof}
\end{proposition}\smallskip

The general $n$-th order Bernstein polynomial is of the form
$$
b(x;n,w)= \sum_{i=1}^n w_{i}\beta(x;i,n-i+1), \ x \in [0,1],
$$
where $w=(w_1,\cdots,w_n)\in \mathbb{R}^n$ and $\beta\left(x;a,b\right)= \frac{\mathrm{\Gamma} (a+b)}{\mathrm{\Gamma}(a)\mathrm{\Gamma}(b)}x^{a-1}{(1-x)}^{b-1},\ x\in[0,1], \ a,b\geqslant 0$. 
\begin{proposition}\label{proposition:B.2}\ \smallskip
For any $w=(w_1,\cdots,w_n)\in \mathbb{R}^n$, we have 
\begin{my_enum}
	\item $\Vert b(\cdot;n,w)\Vert_{\infty}\leqslant n \max_{i=1,2,\cdots,n}|w_{i}|$.\label{proposition:B.2(a)}
	\item $L(b(\cdot;n,w)) \leqslant 2n^2\max_{i=1,2,\cdots,n}|w_{i}|$.\label{proposition:B.2(b)}
\end{my_enum}
\begin{proof}
The first inequality is proved in Lemma 2.2 of \citet[p.~21]{Kruijer}. We will focus on proving the second inequality. By direct calculation, we obtain
$$
b^{\prime}(x;n,w)=n\sum_{i=1}^{n-1}(w_{i+1}-w_{i}) \beta(x;i,n-i)
$$
with the convention $\sum_{i=1}^0=0$. By noticing that $\sum_{i=1}^{n-1}\beta(x;i,n-i)=n-1$, we have
\begin{align*}
&|b^{\prime}(x;k,w)|
\leqslant n\sum_{i=1}^{n-1}|w_{i+1,n}-w_{i,n}| \beta(x;i,n-i) \\
&\leqslant n\max_{i=1,\cdots,n-1}\{|w_{i+1}-w_{i}|\}\sum_{i=1}^{n-1}\beta(x;i,n-i)
=n(n-1)\max_{i=1,\cdots,n-1}\{|w_{i+1}-w_{i}|\} 
\end{align*}
for all $0\leqslant x\leqslant 1$.
Therefore,
\begin{align*}
&L(b(\cdot;n,w))=\left\Vert b^{\prime}(\cdot;n,w) \right\Vert_{\infty}\leqslant n(n-1)\max_{i=1,\cdots,n-1}\{|w_{i+1}-w_{i}|\}\leqslant 2n^2\max_{i=1,2,\cdots,n}|w_{i}|.
\end{align*}
\end{proof}
\end{proposition}

Let $f \in \textrm{Lip}([0,1])$ be a probability density function with corresponding cumulative distribution function $F$, then the $n$-th order (Kantorovich-type) Bernstein polynomial associated with $f$ is defined as
$$
b(x;n,w_n(F))= \sum_{i=1}^n w_{i,n}(F)\beta(x;i,n-i+1), \ x \in [0,1],
$$
where $w_{i,n}(F)=F\left(\frac{i}{n}\right)-F\left(\frac{i-1}{n}\right)$, $i=1,2,\cdots,n$, and \newline
 $w_n(F)= \left(w_{1,n}(F),w_{2,n}(F),\cdots,w_{n,n}(F) \right)$.

\begin{proposition}\label{proposition:B.3} We have
\begin{my_enum}
	\item $\| f - b\left(\cdot;n,w_n(F)\right)\|_{\infty}\leqslant \frac{3L(f)}{\sqrt{n}}.$ \label{proposition:B.3(a)}
	\item If $f$ is continuously differentiable, then $\| f - b\left(\cdot;n,w_n(F)\right)\|_{\mathrm{Lip}}\to 0.$ \label{proposition:B.3(b)}
	\item $\|b\left(\cdot;n,w_n(F)\right) \|_{\mathrm{Lip}}\leqslant \|f \|_{\mathrm{Lip}}$. \label{proposition:B.3(c)}
\end{my_enum}
\begin{proof}
	(a) is proved in Lemma 2.1 of \citet[p.~20]{Kruijer}. (b) is a direct consequence of Theorem 2.4 of \citet{kantorovich_poly}. As for (c), we have by definition
$$
w_{i,n}(F)=F\left(\frac{i}{n}\right)-F\left(\frac{i-1}{n}\right)=\int_{\frac{i-1}{n}}^{\frac{i}{n}}f(x)dx.
$$
Thus, we get $w_{i,n}(F)\leqslant \frac{1}{n}\Vert f \Vert_{\infty}, \ i=1,2,\cdots,n,$ which implies $\|b(\cdot;k,w_n(F))\|_{\infty}\leqslant \Vert f \Vert_{\infty}$ in light of the fact that $\sum_{i=1}^n\beta(x;i,n-i+1)=n$.

As for the bound of Lipschitz constant, because for any $i=1,2,\cdots,n-1$, we have
$$
|w_{i+1,n}(F)-w_{i,n}(F)|\leqslant \int_{\frac{i-1}{n}}^{\frac{i}{n}}\left|f\left( x+\frac{1}{n}\right) - f\left(x\right) \right|dx\leqslant \frac{L(f)}{n^2}.
$$
Therefore, for all $x\in[0,1]$,
$$
|b^{\prime}(x;k,w_n(F))|=n\left|\sum_{i=1}^{n-1}(w_{i+1,n}(F)-w_{i,n}(F)) \beta(x;i,n-i)\right|\leqslant L(f),
$$
which implies $L(b(\cdot;k,w_n(F)))\leqslant L(f)$. Together this yields (c).
\end{proof}
\end{proposition}

\section{Proof of Proposition \ref{proposition:3.2}}\label{appendix:C}
Let 
$$
\tau_{0}=\int_0^1\frac{\varphi_0(x)}{\varphi_{\mathrm{par}}^{\delta}(x)}dx,
$$
and $q_0=\frac{1}{\tau_0}\frac{\varphi_0}{\varphi_{\mathrm{par}}^{\delta}}$. Then $q_0$ is a probability density function on $[0,1]$. Denote by $Q_{0}$ the cumulative distribution function corresponding to $q_{0}$. Our goal is to show that  $\Pi_C\left({B}(\varphi_0,r) \right)>0$ for all $r>0$. To this end, consider
\begin{align*}
	B_{M_1,M_2,M_3}=&\left\{\tau b\left(\cdot;k,w_{k}(G)\right)\varphi_{\mathrm{par}}^{\delta}:\right.\\
&\quad \left. k=M_1,|\tau-\tau_{0}|\leqslant M_2, \max_{i=1,2,\cdots,k}|w_{i,k}(G)-w_{i,k}(Q_{0})|\leqslant M_3  \right\}
\end{align*}
We will show that (a) $\Pi_C(B_{M_1,M_2,M_3})>0$  for all $M_1=1,2,\ldots$ and $M_2,M_3>0$ and (b) that there exist such $M_1,M_2,M_3$ that $B_{M_1,M_2,M_3}\subset B(\varphi_0,r)$  for any $r>0$.

Recall that $\rho$, which is the probability mass function of $K$, has full support. Clearly, we have $\rho(M_1)>0$. Since both the distribution of $\tau$ and the Dirichlet process have full support (see e.g. \citealt[Appendix B.1]{choudhuri_et_al}), the following two sets:
$$
\left\{\tau\in (0,\infty): |\tau-\tau_{0}|\leqslant M_2 \right\}
$$
and
$$
\left\{G: \max_{i=1,2,\cdots,M_1}|w_{i,M_1}(G)-w_{i,M_1}(Q_{0})|\leqslant M_3 \right\}
$$
also have positive probabilities. Therefore, $\Pi_{C}(B_{M_1,M_2,M_2})>0$, showing (a).

For (b), first note that by the triangle inequality, Propositions \ref{proposition:B.2}\eqref{proposition:B.2(a)} and \ref{proposition:B.3}\eqref{proposition:B.3(a)},
\begin{align*}
&\Vert \tau b(\cdot;k,w_k(G))\varphi_{\mathrm{par}}^{\delta} -\tau_0 q_0 \varphi_{\mathrm{par}}^{\delta}\Vert_{\infty} \\
&\leqslant \tau \Vert b(\cdot;k,w_k(G)-w_k(Q_0)) \Vert_{\infty} \Vert \varphi_{\mathrm{par}}\Vert_{\infty}^{\delta}
 + |\tau-\tau_0|\Vert b(\cdot;k,w_k(Q_0)) \Vert_{\infty} \Vert \varphi_{\mathrm{par}}\Vert_{\infty}^{\delta} \\
&\quad + \tau_0 \Vert b(\cdot;k,w_k(Q_0)) - q_0 \Vert_{\infty} \Vert \varphi_{\mathrm{par}}\Vert_{\infty}^{\delta}\\
&\le (\tau_0+M_2)\, M_1\,M_3 \,\|\varphi_{\mathrm{par}}\|^{\delta}_{\infty}+ M_2\, \|q_0\|_{\infty}\|\varphi_{\mathrm{par}}\|^{\delta}_{\infty}+\tau_0 \frac{L(q_0)}{\sqrt{M_1}},
\end{align*}
which by Propositions \ref{proposition:B.1}\eqref{proposition:B.1(a)} can be made arbitrarily small by first choosing an $M_1$ large enough for the last term to become small as well as an $M_2$ small enough for the second term to become small and finally (dependent on $M_1$ and $M_2$) an $M_3$ small enough for the first term to become small. 

This shows in particular that for any $\epsilon>0$ there exists $M_1$ big and $M_2,M_3$ small enough that $\|\varphi-\varphi_0\|_{\infty}<\epsilon$ for any $\varphi\in B_{M_1,M_2,M_3}$ as well as
$\varphi(x)\ge \varphi_0(x)- \|\varphi-\varphi_0\|_{\infty}>m$ for $\epsilon$ small enough by Assumption~\ref{eq_ass_param_space}.  In particular, for any $r>0$ it holds $\|\varphi-\varphi_0\|_{\infty}<r$ and $\varphi$ fulfills the lower bound in \eqref{eq_param_space}.
We will now complete the proof that for sufficiently large $M_1$ and small $M_2,M_3$ it holds that $\|\varphi\|_{\mathrm{Lip}}$ fulfills the upper bound.
Indeed, we get by the triangle inequality, Propositions \ref{proposition:B.1}\eqref{proposition:B.1(b)} and \ref{proposition:B.2}
\begin{align*}
	&\Vert \tau b(\cdot;k,w_k(G))\varphi_{\mathrm{par}}^{\delta} \Vert_{\mathrm{Lip}}\\
	&\leqslant \Vert \tau b(\cdot;k,w_k(Q_0))\varphi_{\mathrm{par}}^{\delta} \Vert_{\mathrm{Lip}} + \Vert \tau b(\cdot;k,w_k(G)-w_k(Q_0))\varphi_{\mathrm{par}}^{\delta} \Vert_{\mathrm{Lip}}\\
	& \leqslant \left( 1+ \frac{M_2}{\tau_0} \right)\, \left\Vert \tau_0 b(\cdot;k,w_k(Q_0))\varphi_{\mathrm{par}}^{\delta} \right\Vert_{\mathrm{Lip}} 
+  3(\tau_0+M_2) M_1^2 M_3 \Vert \varphi_{\mathrm{par}}^{\delta} \Vert_{\mathrm{Lip}}. 
\end{align*}
The second part becomes again arbitrarily small for appropriately chosen $M_j$, $j=1,2,3$, such that the upper bound of \eqref{eq_param_space} can be met as soon as\newline $\Vert \tau_0 b(\cdot;k,w_k(Q_0))\varphi_{\mathrm{par}}^{\delta} \Vert_{\mathrm{Lip}} <M$.

Indeed, by Propositions \ref{proposition:B.1}\eqref{proposition:B.1(b)} and \ref{proposition:B.3}\eqref{proposition:B.3(c)}, we have
\begin{align}\label{eq_param_diff}
	\Vert \tau_0 b(\cdot;k,w_k(Q_0))\varphi_{\mathrm{par}}^{\delta} \Vert_{\mathrm{Lip}}\le \tau_0 \Vert q_0\Vert_{\mathrm{Lip}}\, \Vert \varphi_{\mathrm{par}}^{\delta} \Vert_{\mathrm{Lip}}=\left\Vert \frac{\varphi_0}{\varphi_{\mathrm{par}}^{\delta}} \right\Vert \, \Vert\varphi_{\mathrm{par}}^{\delta} \Vert_{\mathrm{Lip}},
\end{align}
such that the assertion follows from Assumption~\ref{eq_ass_param_space}.

\begin{proof}[Proof of Remark~\ref{rem_param_space}]
	In this case, we can prove \eqref{eq_param_diff} without Assumption~\ref{eq_ass_param_space}, since by the triangle inequality and Proposition \ref{proposition:B.1}\eqref{proposition:B.1(b)},
\begin{align*}
&\Vert \tau_0 b(\cdot;k,w_k(Q_0))\varphi_{\mathrm{par}}^{\delta} \Vert_{\mathrm{Lip}}\leqslant \tau_0\|b(\cdot;k,w_k(Q_0)) - q_0\|_{\mathrm{Lip}}\,\|\varphi_{\mathrm{par}}^{\delta} \|_{\mathrm{Lip}} + \|\varphi_0 \|_{\mathrm{Lip}},
\end{align*}
where the first term becomes arbitrarily small for $M_1=k$ large enough due to Proposition \ref{proposition:B.3}\eqref{proposition:B.3(b)} and the second one is strictly smaller than $M$.
\end{proof}

\section*{Acknowledgements}
RM gratefully acknowledges support by a James Cook Fellowship from Government funding, administered by the Royal Society Te Ap\={a}rangi, and CK and RM support by the DFG Grant KI 1443/3-2.

\bibliographystyle{ba} 
\bibliography{ref}

\begin{thebibliography}{57}
\newcommand{\enquote}[1]{``#1''}
\expandafter\ifx\csname natexlab\endcsname\relax\def\natexlab#1{#1}\fi
\expandafter\ifx\csname url\endcsname\relax
  \def\url#1{{\tt #1}}\fi
\expandafter\ifx\csname urlprefix\endcsname\relax\def\urlprefix{URL }\fi
\ifx\endbibitem\undefined \let\endbibitem\relax\fi

\bibitem[{Aliprantis and Border(2006)}]{aliprantis_and_border}
Aliprantis, C.~D. and Border, K.~C. (2006).
\newblock {\em Infinite Dimensional Analysis: A Hitchhiker’s Guide\/}.
\newblock Springer-Verlag, 3rd edition.
\endbibitem

\bibitem[{Andrews(1984)}]{non_mixing_sequence}
Andrews, D. W.~K. (1984).
\newblock \enquote{Non-Strong Mixing Autoregressive Processes.}
\newblock {\em Journal of Applied Probability\/}, 21(4): 930--934.
\endbibitem

\bibitem[{Belitser and Serra(2014)}]{belitser_and_serra2014}
Belitser, E. and Serra, P. (2014).
\newblock \enquote{Adaptive priors based on splines with random knots.}
\newblock {\em Bayesian Analysis\/}, 9(4): 859--882.
\endbibitem

\bibitem[{Bissiri et~al.(2016)Bissiri, Holmes, and Walker}]{bissiri_et_al}
Bissiri, P.~G., Holmes, C.~C., and Walker, S.~G. (2016).
\newblock \enquote{A general framework for updating belief distributions.}
\newblock {\em Journal of the Royal Statistical Society - Series B\/}, 78(5):
  1103--1130.
\endbibitem

\bibitem[{Bradley(2005)}]{bradley}
Bradley, R.~C. (2005).
\newblock \enquote{Basic properties of strong mixing conditions. A survey and
  some open questions.}
\newblock {\em Probability Surveys\/}, 2: 107--144.
\endbibitem

\bibitem[{Bradley(2007)}]{bradley2007introduction}
--- (2007).
\newblock {\em Introduction to strong mixing conditions\/}, volume~I.
\newblock Kendrick press.
\endbibitem

\bibitem[{Brockwell and Davis(1991)}]{brockwell_and_davis}
Brockwell, P.~J. and Davis, R.~A. (1991).
\newblock {\em Time Series: Theory and Methods\/}.
\newblock Springer-Verlag, 2nd edition.
\endbibitem

\bibitem[{Cadonna et~al.(2017)Cadonna, Kottas, and Prado}]{Cadonna2017}
Cadonna, A., Kottas, A., and Prado, R. (2017).
\newblock \enquote{Bayesian mixture modeling for spectral density estimation.}
\newblock {\em Statistics \& Probability Letters\/}, 125: 189--195.
\endbibitem

\bibitem[{Choi and Schervish(2007)}]{choi_and_schervish}
Choi, T. and Schervish, M.~J. (2007).
\newblock \enquote{On posterior consistency in nonparametric regression
  problems.}
\newblock {\em Journal of Multivariate Analysis\/}, 98: 1969--1987.
\endbibitem

\bibitem[{Chopin et~al.(2013)Chopin, Rousseau, and Liseo}]{Chopin13}
Chopin, N., Rousseau, J., and Liseo, B. (2013).
\newblock \enquote{Computational aspects of {B}ayesian spectral density
  estimation.}
\newblock {\em Journal of Computational and Graphical Statistics\/}, 22(3):
  533--557.
\endbibitem

\bibitem[{Choudhuri et~al.(2004)Choudhuri, Ghosal, and Roy}]{choudhuri_et_al}
Choudhuri, N., Ghosal, S., and Roy, A. (2004).
\newblock \enquote{Bayesian estimation of the spectral density of a time
  series.}
\newblock {\em {Journal of the American Statistical Association}\/}, 99(468):
  1050--1059.
\endbibitem

\bibitem[{Contereras-Crist\'{a}n et~al.(2006)Contereras-Crist\'{a}n,
  Guti\'{e}rrez-Pe\~na, and Walker}]{a_note_on_whittle}
Contereras-Crist\'{a}n, A., Guti\'{e}rrez-Pe\~na, E., and Walker, S.~G. (2006).
\newblock \enquote{A note on Whittle's likelihood.}
\newblock {\em Communications in Statistics - Simulation and Computation\/},
  35(4): 857--875.
\endbibitem

\bibitem[{Dahlhaus et~al.(1996)Dahlhaus, Janas et~al.}]{dahlhaus1996frequency}
Dahlhaus, R., Janas, D., et~al. (1996).
\newblock \enquote{A frequency domain bootstrap for ratio statistics in time
  series analysis.}
\newblock {\em Annals of Statistics\/}, 24(5): 1934--1963.
\endbibitem

\bibitem[{Davidson(1994)}]{davidson}
Davidson, J. (1994).
\newblock {\em Stochastic Limit Theory: An Introduction for Econometricians\/}.
\newblock New York: Oxford University press.
\endbibitem

\bibitem[{Dovgoshey et~al.(2006)Dovgoshey, Martio, Ryazanov, and
  Vuorinen}]{cantor_function}
Dovgoshey, O., Martio, O., Ryazanov, V., and Vuorinen, M. (2006).
\newblock \enquote{The Cantor function.}
\newblock {\em Expositiones Mathematicae\/}, 24: 1--37.
\endbibitem

\bibitem[{Dudley(2004)}]{dudley}
Dudley, R.~M. (2004).
\newblock {\em Real Analysis and Probability\/}.
\newblock Cambridge University Press.
\endbibitem

\bibitem[{Edwards et~al.(2019)Edwards, Meyer, and
  Christensen}]{edwards_et_al2019}
Edwards, M., Meyer, R., and Christensen, N. (2019).
\newblock \enquote{Bayesian nonparametric spectral density estimation using
  B-spline priors.}
\newblock {\em Statistics and Computing\/}, 29: 67--78.
\endbibitem

\bibitem[{Eriksson et~al.(2004)Eriksson, Estep, and
  Johnson}]{applied_math_body_and_soul}
Eriksson, K., Estep, D., and Johnson, C. (2004).
\newblock {\em Applied Mathematics: Body and Soul Volume 1: Derivatives and
  Geometry in IR3\/}.
\newblock Springer-Verlag Berlin Heidelberg, 1 edition.
\endbibitem

\bibitem[{Farouki(2012)}]{bernstein}
Farouki, R.~T. (2012).
\newblock \enquote{The Bernstein polynomial basis: A centennial retrospective.}
\newblock {\em Computer Aided Geometric Design\/}, 29: 379--419.
\endbibitem

\bibitem[{Francq and Zakoian(2019)}]{francq_and_zakoian}
Francq, C. and Zakoian, J. (2019).
\newblock {\em GARCH Models structure, statistical inference and financial
  applications\/}.
\newblock John Wiley \& Sons, 2nd edition.
\endbibitem

\bibitem[{Gangopadhyay et~al.(1999)Gangopadhyay, Mallick, and
  Denison}]{Gango98}
Gangopadhyay, A., Mallick, B., and Denison, D. (1999).
\newblock \enquote{Estimation of spectral density of a stationary time series
  via an asymptotic representation of the periodogram.}
\newblock {\em Journal of statistical planning and inference\/}, 75(2):
  281--290.
\endbibitem

\bibitem[{Gao and Zhou(2016)}]{gao_and_zhou2016}
Gao, C. and Zhou, H. (2016).
\newblock \enquote{Rate exact Bayesian adaptation with modified block priors.}
\newblock {\em The Annals of Statistics\/}, 44(1): 318--345.
\endbibitem

\bibitem[{Ghosal and van~der Vaart(2017)}]{nonpara_bayesian_inference}
Ghosal, S. and van~der Vaart, A. (2017).
\newblock {\em Fundamentals of Nonparametric Bayesian Inference\/}, volume~44.
\newblock Cambridge University Press.
\endbibitem

\bibitem[{Gin\'e and Nickl(2011)}]{gine_and_nickl2011}
Gin\'e, E. and Nickl, R. (2011).
\newblock \enquote{Rates of contraction for posterior distributions in
  $L_r$-metrics, $1\leqslant r \leqslant \infty$.}
\newblock {\em The Annals of Statistics\/}, 39(6): 2883--2911.
\endbibitem

\bibitem[{Gloube and van Loan(1996)}]{matrix_comp}
Gloube, G.~H. and van Loan, C.~F. (1996).
\newblock {\em Matrix Computations\/}.
\newblock Baltimore: The Johns Hopkins University Press, 3rd edition.
\endbibitem

\bibitem[{Gray(2006)}]{gray}
Gray, R.~M. (2006).
\newblock {\em Toeplitz and circulant matrices: A review\/}.
\newblock Now Publishers Inc.
\endbibitem

\bibitem[{Hannan(1973)}]{10.2307/3212501}
Hannan, E.~J. (1973).
\newblock \enquote{The Asymptotic Theory of Linear Time-Series Models.}
\newblock {\em Journal of Applied Probability\/}, 10(1): 130--145.
\newline\urlprefix\url{http://www.jstor.org/stable/3212501}
\endbibitem

\bibitem[{Hermansen(2008)}]{Hermansen08}
Hermansen, G.~H. (2008).
\newblock \enquote{Bayesian nonparametric modelling of covariance functions,
  with application to time series and spatial statistics.}
\newblock Ph.D. thesis, Universitetet i Oslo.
\endbibitem

\bibitem[{Ibragimov and Linnik(1971)}]{ibragimov_and_linnik}
Ibragimov, I.~A. and Linnik, Y.~V. (1971).
\newblock {\em Independent and Stationary Sequences of Random Variables\/}.
\newblock Groningen: Wolters-Noordhoff.
\endbibitem

\bibitem[{Kakizawa(2006)}]{kakizawa2006}
Kakizawa, Y. (2006).
\newblock \enquote{Bernstein polynomial estimation of a spectral density.}
\newblock {\em Journal of Time Series Analysis\/}, 27(2): 253--287.
\endbibitem

\bibitem[{Katznelson(2004)}]{katznelson}
Katznelson, Y. (2004).
\newblock {\em An Introduction to Harmonic Analysis\/}.
\newblock Cambridge University Press, 3rd edition.
\endbibitem

\bibitem[{Kirch et~al.(2019)Kirch, Edwards, Meier, and Meyer}]{kirch_et_al}
Kirch, C., Edwards, M.~C., Meier, A., and Meyer, R. (2019).
\newblock \enquote{Beyond Whittle: Nonparametric correction of a parametric
  likelihood with a focus on Bayesian time series analysis.}
\newblock {\em Bayesian Analysis\/}, 14(4): 1037--1073.
\endbibitem

\bibitem[{Kleijn and van~der Vaart(2006)}]{kleijn_and_van_der_vaart}
Kleijn, B. J.~K. and van~der Vaart, A.~W. (2006).
\newblock \enquote{Misspecification in infinite-dimensional Bayesian
  statistics.}
\newblock {\em The Annals of Statistics\/}, 34(2): 837--877.
\endbibitem

\bibitem[{Kreiss et~al.(2003)Kreiss, Paparoditis
  et~al.}]{kreiss2003autoregressive}
Kreiss, J.-P., Paparoditis, E., et~al. (2003).
\newblock \enquote{Autoregressive-aided periodogram bootstrap for time series.}
\newblock {\em The Annals of Statistics\/}, 31(6): 1923--1955.
\endbibitem

\bibitem[{Kruijer(2008)}]{Kruijer}
Kruijer, W. (2008).
\newblock \enquote{Convergence Rates in Nonparametric Bayesian Density
  Estimation.}
\newblock Ph.D. thesis, Vrije University of Amsterdam.
\endbibitem

\bibitem[{Liseo et~al.(2001)Liseo, Marinucci, and Petrella}]{Liseo01}
Liseo, B., Marinucci, D., and Petrella, L. (2001).
\newblock \enquote{Bayesian semiparametric inference on long-range dependence.}
\newblock {\em Biometrika\/}, 88(4): 1089--1104.
\endbibitem

\bibitem[{Meier(2018)}]{Alex-thesis}
Meier, A. (2018).
\newblock \enquote{A matrix Gamma process and applications to Bayesian analysis
  of multivariate time series.}
\newblock Ph.D. thesis, Otto-von-Guericke-Universität Magdeburg, Fakultät
  für Mathematik.
\newline\urlprefix\url{http://dx.doi.org/10.25673/13407}
\endbibitem

\bibitem[{Meier et~al.(2017)Meier, Kirch, Edwards, and Meyer}]{meier_et_al}
Meier, A., Kirch, C., Edwards, M.~C., and Meyer, R. (2017).
\newblock {\em beyondWhittle: Bayesian Spectral Inference for Stationary Time
  Series\/}.
\newblock R package.
\endbibitem

\bibitem[{Meier et~al.(2020)Meier, Kirch, and Meyer}]{meier_et_al2020}
Meier, A., Kirch, C., and Meyer, R. (2020).
\newblock \enquote{Bayesian Nonparametric Analysis of Multivariate Time Series:
  A Matrix Gamma Process Approach.}
\newblock {\em Journal of Multivariate Analysis\/}, 175.
\endbibitem

\bibitem[{\"{O}zarslan and Duman(2016)}]{kantorovich_poly}
\"{O}zarslan, M.~A. and Duman, O. (2016).
\newblock \enquote{Smoothness Properties of Modified Bernstein-Kantorovich
  Operators.}
\newblock {\em Numerical Functional Analysis and Optimization\/}, 37(1):
  92--105.
\endbibitem

\bibitem[{Pawitan and O'sullivan(1994)}]{pawitan_and_osullivan1994}
Pawitan, Y. and O'sullivan, F. (1994).
\newblock \enquote{Nonparametric spectral density estimation using penalized
  Whittle likelihood.}
\newblock {\em Journal of the American Statistical Association\/}, 89:
  600--610.
\endbibitem

\bibitem[{Peligrad and Wu(2010)}]{peligrad_and_wu}
Peligrad, M. and Wu, W.~B. (2010).
\newblock \enquote{Central limit theorem for Fourier transforms of stationary
  processes.}
\newblock {\em The Annals of Probability\/}, 38(5): 2009--2022.
\endbibitem

\bibitem[{Pourahmadi(2001)}]{pourahmadi}
Pourahmadi, M. (2001).
\newblock {\em Foundations of time series analysis and prediction theory\/}.
\newblock John Wiley \& Sons.
\endbibitem

\bibitem[{Rao and Yang(2020)}]{RaoSuhasiniSubba2020RtGa}
Rao, S.~S. and Yang, J. (2020).
\newblock \enquote{Reconciling the Gaussian and Whittle Likelihood with an
  application to estimation in the frequency domain.}
\endbibitem

\bibitem[{Schmidt(1991)}]{cantor}
Schmidt, K.~D. (1991).
\newblock \enquote{The Cantor set in probability theory.}
\newblock Technical report, Universit\"at Mannheim, Fakult\"at f\"ur Mathematik
  und Informatik.
\endbibitem

\bibitem[{Schwartz(1965)}]{SchwartzLorraine1965OBp}
Schwartz, L. (1965).
\newblock \enquote{On Bayes procedures.}
\newblock {\em Z. Wahrscheinlichkeitstheorie verw Gebiete\/}, 4: 10--26.
\endbibitem

\bibitem[{Searc\`oid(2006)}]{metric_space}
Searc\`oid, M.~O. (2006).
\newblock {\em Metric Space\/}.
\newblock Springer London.
\endbibitem

\bibitem[{Serov(2017)}]{serov}
Serov, V. (2017).
\newblock {\em Fourier Series, Fourier Transform and Their Applications to
  Mathematical Physics\/}.
\newblock Springer, 1st edition.
\endbibitem

\bibitem[{Sethuraman(1994)}]{JayaramSethuraman1994ACDO}
Sethuraman, J. (1994).
\newblock \enquote{A CONSTRUCTIVE DEFINITION OF DIRICHLET PRIORS.}
\newblock {\em Statistica Sinica\/}, 4(2): 639--650.
\endbibitem

\bibitem[{Shalizi(2009)}]{shalizi}
Shalizi, C.~R. (2009).
\newblock \enquote{Dynamics of Bayesian updating with dependent data and
  misspecified models.}
\newblock {\em Electronic Journal of Statistics\/}, 3: 1039--1074.
\endbibitem

\bibitem[{Sriram et~al.(2013)Sriram, Ramamoorthi, and Ghosh}]{sriram_et_al}
Sriram, K., Ramamoorthi, R.~V., and Ghosh, P. (2013).
\newblock \enquote{Posterior consistency of Bayesian quantile regression based
  on the misspecifed asymmetric Laplace density.}
\newblock {\em Bayesian Analysis\/}, 8(2): 479--504.
\endbibitem

\bibitem[{Sykulski et~al.(2019)Sykulski, Olhede, Lilly, Guillaumin, and
  Early}]{Sykulski2017}
Sykulski, A.~M., Olhede, S.~C., Lilly, J.~M., Guillaumin, A.~P., and Early,
  J.~J. (2019).
\newblock \enquote{The de-biased Whittle likelihood.}
\newblock {\em Biometrika\/}, 106: 251--266.
\endbibitem

\bibitem[{Syring(2017)}]{syring17}
Syring, N. (2017).
\newblock \enquote{Gibbs Posterior Distributions: New Theory and Applications.}
\newblock Ph.D. thesis, University of Illinois at Chicago.
\endbibitem

\bibitem[{Syring et~al.(2019)Syring, Hong, and Martin}]{syring_et_al}
Syring, N., Hong, L., and Martin, R. (2019).
\newblock \enquote{Gibbs posterior inference on value-at-risk.}
\newblock {\em Scandinavian Actuarial Journal\/}, 2019(7): 548--557.
\endbibitem

\bibitem[{Tamaki(2008)}]{tamaki2008}
Tamaki, K. (2008).
\newblock \enquote{The Bernstein-von Mises theorem for stationary processes.}
\newblock {\em Journal of the Japan Statistical Society\/}, 38(2): 311--323.
\endbibitem

\bibitem[{Whittle(1957)}]{whittle}
Whittle, P. (1957).
\newblock \enquote{Curve and periodogram smoothing.}
\newblock {\em JRSSB\/}, 38--63.
\endbibitem

\bibitem[{Yaskov(2018)}]{yaskov}
Yaskov, P. (2018).
\newblock \enquote{LLN for quadratic forms of long memory time series and its
  applications in random matrix theory.}
\newblock {\em Journal of Theoretical Probability\/}, 31: 2032--2055.
\endbibitem

\end{thebibliography}

\end{document}